\documentclass{stml-l}
\usepackage{amssymb,amsmath,graphics,amscd}
\makeindex
\newtheorem{theorem}{Theorem}[chapter]
\newtheorem{lem}[theorem]{Lemma}
\newtheorem{prop}[theorem]{Proposition}
\newtheorem{cor}[theorem]{Corollary}

\newtheorem{defn}[theorem]{Definition}

\theoremstyle{remark}

\newtheorem{ex}{Ex.}[chapter]
\numberwithin{equation}{chapter}

\newcommand{\half}{1/2}

\newcommand{\ra}{\rightarrow}

\newcommand{\C}{{\mathbb{C}}}
\newcommand{\HH}{{\mathbb{H}}}

\newcommand{\Ad}{\text{Ad}}
\newcommand{\lb}{\langle}
\newcommand{\rb}{\rangle}
\newcommand{\mg}{\mathfrak{g}}
\newcommand{\mh}{\mathfrak{h}}
\newcommand{\mk}{\mathfrak{k}}
\newcommand{\ms}{\mathfrak{s}}

\newcommand{\R}{\mathbb{R}}
\newcommand{\diag}{\text{diag}}

\newcommand{\ii}{\mathbf{i}}
\newcommand{\jj}{\mathbf{j}}
\newcommand{\kk}{\mathbf{k}}
\newcommand{\ad}{\text{ad}}
\newcommand{\ml}{\mathfrak{l}}

\begin{document}

\title{Roots of a Compact Lie Group}

\author {Kristopher Tapp}

\maketitle
This expository article introduces the topic of roots in a compact Lie group.  Compared to the many other treatments of this standard topic, I intended for mine to be relatively elementary, example-driven, and free of unnecessary abstractions.  Some familiarity with matrix groups and with maximal tori is assumed.

This article is self-contained, but is also intended to serve as a supplemental $10^\text{th}$ chapter of an eventual new edition of my textbook, ``Matrix Groups for Undergraduates'' (AMS, 2005).  All external references are to chapters 1-9 of this textbook.
\setcounter{chapter}{9}
\chapter{Roots}
By the classification theorem, the ``classical'' compact Lie groups, $SO(n),SU(n),$ and $Sp(n)$, together with the five exceptional groups, form the building blocks of all compact Lie groups.  In this chapter, we will use roots to better understand the Lie bracket operation in the Lie algebra, $\mg$, of a classical or general compact Lie group $G$.

Let $\tau$ denote the Lie algebra of a maximal torus, $T$, of $G$.  The roots of $G$ will be defined as a finite collection of linear functions from $\tau$ to $\R$ which together determine all brackets $[X,V]$ with $X\in\tau$ and $V\perp\tau$.  We will eventually discover that the roots determine \emph{all} brackets in $\mg$, so roots provide an extremely useful method of encoding and understanding the entire bracket operation in $\mg$.

This chapter is organizes as follows.  First we will explicitly describe the roots and the bracket operation for $G=SU(n)$.  Next we will define and study the roots of an arbitrary compact Lie group.  We will then apply this general theory to describe the roots and the bracket operation for $G=SO(n)$ and $G=Sp(n)$.  We will then define the ``Weil Group'' of $G$,  and  roughly indicate how the theory of roots leads to a proof of the classification theorem.  Finally, we will define the ``complexification'' of a Lie algebra, to build a bridge between this book and more advanced books which typically emphasize roots of a complexified Lie algebra.

\section{The structure of $\mg=su(n)$}
Let $G=SU(n)$, so $\mg=su(n)$.  In this case, recall that the Lie algebra of the standard maximal torus equals:
$$\tau=\{\diag(\lambda_1\ii,...,\lambda_n\ii)\mid \lambda_1+\cdots+\lambda_n=0\}.$$
For each pair $(i,j)$ of distinct integers between $1$ and $n$, let $H_{ij}\in\tau$ denote the matrix with $\ii$ in position $(i,i)$ and $-\ii$ in position $(j,j)$. Let $E_{ij}$ denote the matrix in $\mg$ with $1$ in position $(i,j)$ and $-1$ in position $(j,i)$.  Let $F_{ij}$ denote the matrix in $\mg$ with $\ii$ in positions $(i,j)$ and  $(j,i)$.  Notice $\{H_{12},H_{23},...,H_{(n-1)n}\}$ is a basis for $\tau$, and that the $E$'s and $F's$ for which $i<j$ together form a basis of $\tau^\perp$.  We've arrived at a basis for $\mg$, which in the case $n=3$ looks like:
\scalebox{0.75}{\parbox{\textwidth}{
\begin{align*}
H_{12} &= \left(\begin{matrix} \ii & 0 & 0 \\ 0 & -\ii & 0 \\ 0 & 0 & 0\end{matrix}\right),
& E_{12} &= \left(\begin{matrix} 0 & 1 & 0 \\ -1 & 0 & 0 \\ 0 & 0 & 0\end{matrix}\right),
& E_{23} &= \left(\begin{matrix} 0 & 0 & 0 \\ 0 & 0 & 1 \\ 0 & -1 & 0\end{matrix}\right),
& E_{13} &= \left(\begin{matrix} 0 & 0 & 1 \\ 0 & 0 & 0 \\ -1 & 0 & 0\end{matrix}\right), \\
H_{23} &= \left(\begin{matrix} 0 & 0 & 0 \\ 0 & \ii & 0 \\ 0 & 0 & -\ii\end{matrix}\right),
& F_{12} &= \left(\begin{matrix} 0 & \ii & 0 \\ \ii & 0 & 0 \\ 0 & 0 & 0\end{matrix}\right),
& F_{23} &= \left(\begin{matrix} 0 & 0 & 0 \\ 0 & 0 & \ii \\ 0 & \ii & 0\end{matrix}\right),
& F_{13} &= \left(\begin{matrix} 0 & 0 & \ii \\ 0 & 0 & 0 \\ \ii & 0 & 0\end{matrix}\right).
\end{align*}}}

The bracket of any pair of these basis elements is given by Table~\ref{tablesu3}.
\begin{table}[hb]\begin{center}\begin{tabular}{ | c || c | c || c | c || c | c || c | c |} \hline
   $\mathbf{[\cdot,\cdot]}$ & $\mathbf{H_{12}}$ & $\mathbf{H_{23}}$ & $\mathbf{E_{12}}$ & $\mathbf{F_{12}}$ & $\mathbf{E_{23}}$ & $\mathbf{F_{23}}$ & $\mathbf{E_{13}}$ & $\mathbf{F_{13}}$ \\ \hline\hline
   $\mathbf{H_{12}}$ & $0$ & $0$ & $2F_{12}$ & $-2E_{12}$ & $-F_{23}$ & $E_{23}$ & $F_{13}$ & $-E_{13}$ \\ \hline
   $\mathbf{H_{23}}$ & * & $0$ & $-F_{12}$ & $E_{12}$ & $2F_{23}$ & $-2E_{23}$ & $F_{13}$ & $-E_{13}$ \\ \hline\hline
   $\mathbf{E_{12}}$ & * & * & $0$ & $2H_{12}$ & $E_{13}$ & $F_{13}$ & $-E_{23}$ & $-F_{23}$ \\ \hline
   $\mathbf{F_{12}}$ & * & * & * & $0$ & $F_{13}$ & $-E_{13}$ & $F_{23}$ & $-E_{23}$ \\ \hline\hline
   $\mathbf{E_{23}}$ & * & * & * & * & $0$ & $2H_{23}$ & $E_{12}$ & $F_{12}$ \\ \hline
   $\mathbf{F_{23}}$ & * & * & * & * & * & $0$ & $-F_{12}$ & $E_{12}$ \\ \hline\hline
   $\mathbf{E_{13}}$ & * & * & * & * & * & * & $0$ & $2H_{13}$ \\ \hline
   $\mathbf{F_{13}}$ & * & * & * & * & * & * & * & $0$ \\ \hline
\end{tabular}
\caption{The Lie bracket operation for $\mg=su(3)$}\label{tablesu3}
\end{center}
\end{table}

The *'s below the diagonal remind us that these entries are determined by those above the diagonal, since $[A,B]=-[B,A]$.

Our goal is to summarize the important patterns in Table~\ref{tablesu3} and their generalizations to $\mg=su(n)$.  First, define $\ml_{ij}:=\text{span}\{E_{ij},F_{ij}\}$, so we have an orthogonal direct sum:
$$su(3) = \tau\oplus\ml_{12}\oplus\ml_{23}\oplus\ml_{13}.$$
Here, ``orthogonal direct sum,'' denoted with the ``$\oplus$'' symbol, means that the spaces are mutually orthogonal and together span $su(3)$.  For $\mg=su(n)$, we have the analagous orthogonal direct sum:
$$su(n) = \tau\oplus\{\ml_{ij}\mid 1\leq i<j\leq n\}.$$
The spaces $\ml_{ij}$ are called the \underline{root spaces} of $SU(n)$.

The first two rows of Table~\ref{tablesu3} show that for each pair $(i,j)$, the space $\ml_{ij}$ is \underline{$\ad_{\tau}$-invariant}.  This means that for each $X\in\tau$ and each $V\in\ml_{ij}$, we have $\ad_X(V):=[X,V]\in\ml_{ij}$.  That is,
$\ad_\tau(\ml_{ij})\subset\ml_{ij}.$  More generally, for $\mg=su(n)$, each $\ml_{ij}$ is $\ad_\tau$-invariant.

Choose a fixed pair $(i,j)$.  How do matrices in $\ml_{ij}$ bracket with each other?  How do they bracket with elements of $\tau$?  These two questions are related, and their relationship is the key to understanding roots.  The answer to the first question is: $[E_{ij},F_{ij}]=2H_{ij}$.  It is useful to normalize our basis of $\ml_{ij}$, and report this answer as: \begin{equation}\label{alha}\hat\alpha_{ij}:=\left[\frac{E_{ij}}{|E_{ij}|},\frac{F_{ij}}{|F_{ij}|}\right]=H_{ij}.\end{equation}

The answer to the second question is that for all $X\in\tau$, we have:
$$[X,E_{ij}]=\alpha F_{ij}, \text{ and } [X,F_{ij}]=-\alpha E_{ij},$$
for some $\alpha\in\R$, which we write as $\alpha=\alpha_{ij}(X)$ to point out that it depends on $X$ and on the pair $(i,j)$.  From the first two rows of Table~\ref{tablesu3}, we see:
\begin{align*}
\alpha_{12}(H_{12}) &= 2, & \alpha_{23}(H_{12}) &= -1, & \alpha_{13}(H_{12})&= 1, \\
\alpha_{12}(H_{23}) &= -1, & \alpha_{23}(H_{23}) &= 2, & \alpha_{13}(H_{23})&= 1.
\end{align*}
By linearlity, the values $\alpha_{ij}(H_{12})$ and $\alpha_{ij}(H_{23})$ determine $\alpha_{ij}(X)$ for any $X\in\tau$.  For example, with $X=H_{13}=H_{12}+H_{23}$, we add the above two rows, getting:
\begin{align*}
\alpha_{12}(H_{13}) &= 1, & \alpha_{23}(H_{13}) &= 1, & \alpha_{13}(H_{13})&= 2.
\end{align*}
What is the pattern?  The most concise answer involves the matrices $\hat\alpha_{ij}$ defined in Equation~\ref{alha}; namely, for all $X\in\tau$ we have:
$$\alpha_{ij}(X) = \lb \hat\alpha_{ij},X\rb_\R.$$
Recall that $X,\hat\alpha_{ij}\in su(3)\subset M_3(\C)\cong\C^9\cong\R^{18}$, and $\lb \hat\alpha_{ij},X\rb_\R$ denotes the standard inner product on $\R^{18}$.  In particular,
$$\lb\text{diag}(\lambda_1\ii,\lambda_2\ii,\lambda_3\ii),\text{diag}(\mu_1\ii,\mu_2\ii,\mu_3\ii)\rb_\R = \lambda_1\mu_1+\lambda_2\mu_2+\lambda_3\mu_3.$$
For example, if $X=\diag(7\ii,5\ii,-12\ii)$, then $\alpha_{12}(X) = 7-5 = 2$, $\alpha_{23}(X) = 5-(-12)=17$ and $\alpha_{13}(X) = 7-(-12) = 19$.  So we know, for example, that $[X,E_{13}]=19\cdot F_{13}$ and $[X,F_{13}]=-19\cdot E_{13}$.

Everything above generalizes to $\mg=su(n)$.  In particular, for each pair $(i,j)$, the matrix $\hat\alpha_{ij}:=\left[\frac{E_{ij}}{|E_{ij}|},\frac{F_{ij}}{|F_{ij}|}\right]$ (which equals $H_{ij}\in\tau$) determines how any $X\in\tau$ brackets with any element of $\ml_{ij}$, via:
$$[X,E_{ij}] = \alpha_{ij}(X)\cdot F_{ij}, \text{ and } [X,F_{ij}] = -\alpha_{ij}(X)\cdot E_{ij},$$
with $\alpha_{ij}(X):=\lb\hat\alpha_{ij},X\rb_\R$.

The functions $\{\alpha_{ij}\}$ are called the \underline{roots} of $SU(n)$.  Each root is a linear function from $\tau$ to $\R$.  The corresponding matrices $\{\hat\alpha_{ij}\}$ are called the \underline{dual roots} of $SU(n)$; they encode the same information as the roots.  Their definitions:
$$\hat\alpha_{ij}:=\left[\frac{E_{ij}}{|E_{ij}|},\frac{F_{ij}}{|F_{ij}|}\right]\,\,\,\text{ and }\,\,\, \alpha_{ij}(X):=\lb\hat\alpha_{ij},X\rb_\R$$
make sense for any indices $i\neq j$ (because $E_{ij}$ and $F_{ij}$ make sense).  The ones for which $i<j$ are called \underline{positive} roots and dual roots; these were the most relevant in the above discussion.  Notice that $\hat\alpha_{ij}=-\hat\alpha_{ji}$ (because $E_{ij}=-E_{ji}$ and $F_{ij}=F_{ji}$), so $\alpha_{ij}=-\alpha_{ji}$.

Lastly, we wish to understand how matrices in one of the $\ml$'s bracket with matrices in another.  For $G=SU(3)$, Table~\ref{tablesu3} shows that $[\ml_{12},\ml_{23}]\subset\ml_{13}$, $[\ml_{12},\ml_{13}]\subset\ml_{23}$, and $[\ml_{23},\ml_{13}]\subset\ml_{12}$.  For $G=SU(n)$, $[\ml_{ij},\ml_{jk}]\subset\ml_{ik}$, with individual brackets given by:
\begin{table}[h!]\begin{center}\begin{tabular}{ | c || c | c |} \hline
   $\mathbf{[\cdot,\cdot]}$ & $\mathbf{E_{jk}}$ & $\mathbf{F_{jk}}$  \\ \hline\hline
   $\mathbf{E_{ij}}$ & $E_{ik}$ & $F_{ik}$  \\ \hline
   $\mathbf{F_{ij}}$ & $F_{ik}$ & $-E_{ik}$ \\ \hline
\end{tabular}
\caption{The bracket $[\ml_{ij},\ml_{jk}]\subset\ml_{ik}$ in $su(n)$}\label{ijjkbrak}
\end{center}
\end{table}

Notice that $\ml_{ij}=\ml_{ji}$, but the table is arranged with the second index of the first $\ml$ equalling the first index of the second.  To translate Table~\ref{ijjkbrak} into the block of Table~\ref{tablesu3} corresponding to  $[\ml_{12},\ml_{13}]\subset\ml_{23}$ or $[\ml_{23},\ml_{13}]\subset\ml_{12}$, just use that $E_{ij}=-E_{ji}$ and $F_{ij}=F_{ji}$ as needed.  For $su(n)$ with $n>3$, there are pairs of $\ml$'s which share no common index.  These pairs bracket to zero.  For example, $[\ml_{12},\ml_{34}]=0$.

To foreshadow the general theory, we mention that the dual roots encode the above information about which pairs of $\ml$'s bracket into which.  Consider:
\begin{equation}\label{dictionalry}[\ml_{ij},\ml_{jk}] \subset \ml_{ik}\,\,\, \longleftrightarrow\,\,\, \hat\alpha_{ij} + \hat\alpha_{jk} = \hat\alpha_{ik}.\end{equation}
Equation~\ref{dictionalry} can translate facts about brackets of $\ml$'s into facts about sums of dual roots.  If you wish to work only with positive dual roots, you may need to introduce negative signs, as in:
$$[\ml_{23},\ml_{13}] \subset \ml_{12}\,\,\, \longleftrightarrow\,\,\, \hat\alpha_{23} - \hat\alpha_{13} = -\hat\alpha_{12},$$
obtained from Equation~\ref{dictionalry} via: $[\ml_{23},\ml_{13}]=[\ml_{23},\ml_{31}] = \ml_{21}$.
Also, the fact that $[\ml_{12},\ml_{34}]=0$ translates into the fact that $\hat\alpha_{12}\pm\hat\alpha_{34}$ is not a dual root.

In summary, the bracket $[\ml_{ab},\ml_{cd}]$ can be found as follows.  If either $\hat\alpha_{ab}+\hat\alpha_{cd}$ or $\hat\alpha_{ab}-\hat\alpha_{cd}$ equals a dual root, $\hat\alpha_{ij}$, (they never both equal a dual root) then $[\ml_{ab},\ml_{cd}]\subset\ml_{ij}$.  Otherwise, $[\ml_{ab},\ml_{cd}]=0$.

\section{An invariant decomposition of $\mg$}
Let $G$ be compact Lie group with Lie algebra $\mg$.  Let $T\subset G$ be a maximal torus, with Lie algebra $\tau\subset\mg$.  In this and the next two sections, we generalize to $G$ all of the structures and patterns that we previously observed for $SU(n)$.  We will require:
\begin{prop}There exists an $\Ad$-invariant inner product, $\lb\cdot,\cdot\rb$, on $\mg$.\end{prop}

An \underline{inner product} means a function that associates a real number to each pair of vectors, satisfying the properties of the standard inner product on $\R^n$ enumerated in Proposition 3.3.  This is exactly the structure needed to define norms (as in Definition 3.1) and angles (as in Equation 3.4).  Recall the $\Ad$-invariant inner product for the classical groups, previously denoted $\lb\cdot,\cdot\rb_\R$.  It arose by identifying an $n\times n$ real, complex or quaternionic matrix with $\R^{n^2}$, $\R^{2n^2}$ or $\R^{4n^2}$, and using the standard inner product on this Euclidean space.  An equivalent definition, $\lb X,Y\rb_\R := \text{Real}(X\cdot Y^*)$, was given in the proof of Proposition 8.12.  Recall that ``$\Ad$-invariant'' means  $\lb\Ad_g X,\Ad_gY\rb = \lb X,Y\rb$ for all $g\in G$ and all $X,Y\in\mg$.  As in Proposition~8.14, this implies ``infinitesimal $\Ad$-invariance'': for all $A,B,C\in\mg$,
\begin{equation}\lb[A,B],C\rb = -\lb[A,C],B\rb.\end{equation}

We will require the following general facts about maximal tori, which we previously proved at least for the classical groups.  Recall from Exercise~7.6 that ``$G^0$'' denotes the identity component of $G$.
\begin{prop}[Summary of properties of a maximal torus]\label{summary}\hspace{.1in}
\begin{enumerate}
\item For every $x\in G^0$ there exists $g\in G^0$ such that $x\in g\cdot T\cdot g^{-1}$.
\item Every maximal torus of $G$ equals $g\cdot T\cdot g^{-1}$ for some $g\in G^0$.
\item If $x\in G^0$ commutes with every element of $T$, then $x\in T$.
\item If $X\in\mg$ commutes with every element of $\tau$, then $X\in\tau$.
\end{enumerate}
\end{prop}

We now begin to generalize to $G$ the patterns we have observed for $SU(n)$, beginning with:
\begin{theorem}\label{decompp} $\mg$ decomposes as an orthogonal direct sum,
$$\mg = \tau\oplus\ml_1\oplus\ml_2\oplus\cdots\oplus\ml_m,$$
where each $\ml_i$ is a 2-dimensional \underline{$\Ad_T$-invariant} subspace of $\mg$.
\end{theorem}
The spaces $\{\ml_i\}$ are called the \underline{root spaces} of $G$.  Each root space is \underline{$\Ad_T$-invariant}, which means that for each $g\in T$ and each $V\in\ml_i$, we have $\Ad_g(V)\in\ml_i$.  In other words, $\Ad_T(\ml_i)\subset\ml_i.$  By the definition of the Lie bracket, this implies that each root space $\ml_i$ is also \underline{$\ad_\tau$-invariant}, which means that $\ad_\tau(\ml_i)\subset\ml_i.$

Since $m=\half(\text{dim}(G)-\text{rank}(G)),$ the theorem implies that $\text{dim}(G)-\text{rank}(G)$ is even.

\begin{proof}
For each $g\in T$, the linear function $\Ad_g:\mg\ra\mg$ restricts to $\tau$ as the identity function, because $T$ is abelian.  Therefore, $\Ad_g$ sends any vector $V\in\tau^\perp$ to another vector in $\tau^\perp$, since for all $X\in\tau$,
$$\lb \Ad_g V,X\rb = \lb V,\Ad_{g^{-1}} X\rb = \lb V,X\rb = 0.$$
We choose a fixed orthonormal basis, $\mathcal{B}$, of $\tau^\perp$, via which we identify $\tau^\perp\cong\R^{s}$, where
$s := \text{dim}(\tau^\perp) = \text{dim}(G)-\text{rank}(G).$  For each $g\in T$, the map $\Ad_g:\tau^\perp\ra\tau^\perp$ can be represented with respect to $\mathcal{B}$ as left multiplication by some matrix in $O(s)$; in this way, we can consider $\Ad$ as a smooth homomorphism $\Ad:T\ra O(s)$.

Since $T$ is a compact abelian path-connected Lie group, so must be its image under a smooth homomorphism.  Theorem~9.5 generalizes to say that any compact abelian path-connected Lie group is isomorphic to a torus.  Thus, the image, $\Ad(T)\subset O(s)$, is a torus in $O(s)$.  Let $\tilde{T}$ be a maximal torus of $O(s)$ which contains $\Ad(T)$.  By Proposition~\ref{summary}.2, $\tilde{T}$ equals a conjugate of the standard maximal torus of $O(s)$. Said differently, after conjugating our basis $\mathcal{B}$, we can assume that $\tilde{T}$ equals the standard maximal torus of $O(s)$.

If $s$ is even, so that $s=2m$ for some integer $m$, then we'll write this newly conjugated orthonormal basis of $\tau^\perp$ as
$$\mathcal{B} = \{E_1,F_1,E_2,F_2,...,E_m,F_m\}.$$
Our description in Chapter 9 of the standard maximal torus of $O(2m)$ shows that each of the spaces $\ml_i:=\text{span}\{E_i,F_i\}$ is invariant under the left-multiplication map, $L_a$, for all $a\in\tilde{T}$.  This means that $L_a(\ml_i)\subset\ml_i$.  In particular, this holds for all $a\in\tilde{T}$ for which $L_a$ represents $\Ad_g$ for some $g\in T$.  Therefore, each $\ml_i$ is $Ad_T$-invariant.

It remains to demonstrate that $s$ cannot be odd.  If $s$ were odd, so that $s=2m+1$, then one more element, $V\in\tau^\perp$, would need to be added to the above basis $\mathcal{B}$.  It follows from our description in Chapter 9 of the standard maximal torus of $O(2m+1)$ that $L_a(V)=V$ for all $a\in\tilde{T}$, so $Ad_g(V) = V$ for all $g\in T$.  Therefore, $[X,V]=0$ for all $X\in\tau$, contradicting Proposition~\ref{summary}.4.
\end{proof}

\section{The definition of roots and dual roots}
Decompose $\mg=\tau\oplus\ml_1\oplus\cdots\oplus\ml_m$, as in Theorem~\ref{decompp}.  For each $i$, let $\{E_i,F_i\}$ be an  orthonormal ordered basis for $\ml_i$.

\begin{defn}
For each $i$, define $\hat\alpha_i:=[E_i,F_i]$, and define the linear function $\alpha_i:\tau\ra\R$ such that for all $X\in\tau$,
$\alpha_i(X)=\lb\hat\alpha_i,X\rb$.
\end{defn}

Notice that $\alpha_i$ and $\hat\alpha_i$ contain the same information.  The next proposition shows that the $\alpha$'s determine how vectors in $\tau$ bracket with vectors in the root spaces.
\begin{prop}\label{EiFi}  For each $i$, $\hat\alpha_i\in\tau$.  Further, for all $X\in\tau$,
$$[X,E_i]=\alpha_i(X)\cdot F_i, \text{ and } [X,F_i] = -\alpha_i(X)\cdot E_i.$$
\end{prop}
\begin{proof} Let $X\in\tau$.  Since $\ml_i$ is $\ad_\tau$-invariant, we know $[X,E_i]\in\ml_i$.  Also, $[X,E_i]$ is orthogonal to $\ml_i$'s first basis vector, $E_i$, since
$$\lb[X,E_i],E_i\rb = -\lb[E_i,X],E_i\rb = \lb[E_i,E_i],X\rb = 0,$$
so $[X,E_i]$ must be a multiple of $\ml_i$'s second basis vector, $F_i$.  That is, $[X,E_i]=\lambda\cdot F_i$.  This multiple is:
$$\lambda=\lb[X,E_i],F_i\rb = -\lb[E_i,X],F_i\rb = \lb[E_i,F_i],X\rb = \lb\hat\alpha_i,X\rb = \alpha_i(X).$$
Similarly, $[X,F_i]=-\alpha_i(X)\cdot E_i$.

Finally, we prove that $\hat\alpha_i=[E_i,F_i]\in\tau$.  Using the Jacobi identity (Prop. 8.4) we have for all $X\in\tau$ that:
\begin{eqnarray*}
[X,[E_i,F_i]] & = & [E_i,[X,F_i]] - [F_i,[X,E_i]] \\
              & = & -[E_i,\alpha_i(X)\cdot E_i] - [F_i,\alpha(X)\cdot F_i] = 0.\end{eqnarray*}
Since $[E_i,F_i]$ commutes with every $X\in\tau$, we know $[E_i,F_i]\in\tau$.
\end{proof}

For each $i$, the linear function $\alpha_i:\tau\ra\R$ records the initial speed at which each vector $X\in\tau$ rotates the root space $\ml_i$.  To understand this remark, notice that the function $\ad_X:\ml_i\ra\ml_i$ (which sends $A\mapsto [X,A]$) is given with respect to the ordered basis $\{E_i,F_i\}$ as left multiplication by the ``infinitesimal rotation matrix'':
$$\ad_X = \left(\begin{matrix}0 & -\alpha_i(X) \\ \alpha_i(X) & 0\end{matrix}\right).$$
Thus, the function $Ad_{e^{tX}}:\ml_i\ra\ml_i$ is given in this ordered basis as left multiplication by the rotation matrix:
$$\Ad_{e^{tX}} = e^{\ad_X} = \left(\begin{matrix}\cos (\alpha_i(X) t) & -\sin (\alpha_i(X) t) \\ \sin (\alpha_i(X) t) & \cos (\alpha_i(X) t)\end{matrix}\right).$$

For each index $i$, let $R_i:\ml_i\ra\ml_i$ denote a $90^\circ$ rotation of $\ml_i$ which is ``counterclockwise'' with respect to the ordered basis $\{E_i,F_i\}$.  That is, $R_i$ is the linear function for which $R_i(E_i)=F_i$ and $R_i(F_i)=-E_i$.  Notice that for all $V\in\ml_i$ and all $X\in\tau$, we have:
\begin{gather*}\ad_X(V) = \alpha_i(X)\cdot R_i(V),\\
\Ad_{e^{tX}}V = \cos(\alpha_i(X)t)\cdot V + \sin(\alpha_i(X)t)\cdot R_i(V).
\end{gather*}
simply because these equations are true on a basis $V\in\{E_i,F_i\}$.  More generally, since $\mg=\tau\oplus\ml_1\oplus\cdots\oplus\ml_m$
is an orthogonal decomposition, each $V\in\mg$ decomposes uniquely as  $V=V^0+V^1+\cdots+V^m$, with $V^0\in\tau$ and with $V^i\in\ml_i$ for each $1\leq i\leq m$.  For each $X\in\tau$, $\ad_X$ and $\Ad_{e^{tX}}$ act independently on the $\ml$'s, so:
\begin{gather}
\ad_XV = \sum_{i=1}^m \alpha_i(X)\cdot R_i(V^i), \label{Rotatead}\\
\Ad_{e^{tX}} V = V^0 + \sum_{i=1}^m \cos(\alpha_i(X)t)\cdot V^i + \sin(\alpha_i(X)t)\cdot R_i(V^i).\label{RotateAdad}
\end{gather}

Thus, the one-parameter group $t\mapsto\Ad_{e^{tX}}$ independently rotates each $\ml_i$ with period $2\pi/|\alpha_i(X)|$.  The rotation is counterclockwise if $\alpha_i(X)>0$ and clockwise if $\alpha_i(X)<0$.

We caution that there is generally no basis-independent notion of clockwise.  If we replace the ordered basis $\{E_i,F_i\}$ with $\{E_i,-F_i\}$ (or with $\{F_i,E_i\}$), this causes $R_i$ and $\alpha_i$ to each be multiplied by $-1$, so our notion of clockwise is reversed.   Nevertheless, this sign ambiguity is the only sense in which our definition of $\alpha_i$ is basis-dependent.  The absolute value (or equivalently the square) of each $\alpha_i$ is basis-independent:
\begin{prop}\label{independentsign} Each function $\alpha_i^2:\tau\ra\R^{\geq 0}$ is independent of the choice of ordered orthonormal basis $\{E_i,F_i\}$ for $\ml_i$.
\end{prop}
\begin{proof}
Let $X\in\tau$. Consider the linear function $\ad_X^2:\ml_i\ra\ml_i$, which sends $V\ra\ad_X(\ad_X(V))=[X,[X,V]]$.  By Proposition~\ref{EiFi}, we have $\ad_X^2(E_i) = -\alpha_i(X)^2\cdot E_i$ and $\ad_X^2(F_i) = -\alpha_i(X)^2\cdot F_i$.  Thus,
$$\ad_X^2 = -\alpha_i(X)^2\cdot\text{Id}.$$
Therefore $-\alpha_i(X)^2$ is an eigenvalue of $\ad_X^2$, and so is basis-independent.
\end{proof}

For our general definition, we will use:

\begin{defn}\label{D:dualroot}
A nonzero linear function $\alpha:\tau\ra\R$ is called a \underline{root} of $G$ if there exists a 2-dimensional subspace $\ml\subset\mg$ with an ordered orthonormal basis $\{E,F\}$ such that for each $X\in\tau$, we have
$$[X,E] = \alpha(X)\cdot F\,\,\text{ and }\,\,[X,F] = -\alpha(X)\cdot E.$$
In this case, $\ml$ is called the \underline{root space} for $\alpha$, and the \underline{dual root} for $\alpha$ means the unique vector $\hat\alpha\in\tau$ such that $\alpha(X)=\lb \hat\alpha,X\rb$ for all $X\in\tau$.
\end{defn}

The fact that a unique such dual root vector always exists is justified in Exercise~\ref{dualrootex}.

Notice that if $\alpha$ is a root of $G$ with root space $\ml=\text{span}\{E,F\}$, then $-\alpha$ is a root of $G$ with the same root space $\ml=\text{span}\{F,E\}$.

\begin{prop} The functions $\pm\alpha_1,...,\pm\alpha_m$ are all roots of $G$.
\end{prop}
\begin{proof}
It only remains to show that each $\alpha_i$ is nonzero (not the zero function).  But if $\alpha_i(X)=0$ for all $X\in\tau$, then $E_i$ and $F_i$ would commute with every element of $\tau$, contradicting Proposition~\ref{summary}.4.
\end{proof}

Typically, $m>\text{dim}(\tau)$, so the set $\{\hat\alpha_1,...,\hat\alpha_m\}$ is too big to be a basis of $\tau$, but we at least have:
\begin{prop}\label{rootsspan}
If the center of $G$ is finite, then the dual roots of $G$ span $\tau$.
\end{prop}
\begin{proof}
If some $X\in\tau$ were orthogonal to all of the dual roots, then $[X,A]=0$ for all $A\in\mg$, and therefore, $e^{tA}$ would lie in the center of $G$ for all $t\in\R$.
\end{proof}
Recall that $SO(n)$ (when $n>2$), $SU(n)$ and $Sp(n)$ have finite centers according to Proposition~9.10.  Even though the dual roots are typically linearly dependent, we at least have the following result, whose proof requires representation theory arguments:

\begin{lem}\label{notpar} No pair of the dual roots $\{\hat\alpha_1,...,\hat\alpha_m\}$ are equal (or even parallel) to each other.
\end{lem}

We require this lemma to prove that $\{\pm\alpha_1,...,\pm\alpha_m\}$ are the only roots.  You will observe in Exercise~\ref{notuniqueroots} that if the lemma were false, then there would be other roots.

\begin{defn} A vector $X\in\tau$ is called a \underline{strongly regular vector} if the following are distinct non-zero numbers: $\alpha_1(X)^2,...,\alpha_m(X)^2$.
\end{defn}

For example, when $G=SU(n)$, $X=\diag(\lambda_1\ii,...,\lambda_n\ii)$ is strongly regular if and only if no difference of two $\lambda$'s equals zero or equals the difference of another two $\lambda$'s.
\begin{prop}\label{regularopendense} The strongly regular vectors of $G$  form an open dense subset of $\tau$.  In particular, strongly regular vectors exist.
\end{prop}
\begin{proof}\label{regprop}
Exercise~\ref{regvector}, using Lemma~\ref{notpar}.
\end{proof}

\begin{prop}\label{regvec} If $X\in\tau$ is strongly regular, then the map $\ad_X^2:\mg\ra\mg$ has eigenvalues $0, -\alpha_1(X)^2,-\alpha_2(X)^2,...,-\alpha_m(X)^2$ with corresponding eigenspaces  $\tau,\ml_1,\ml_2,...,\ml_m$.
\end{prop}

This proposition follows from the proof of Proposition~\ref{independentsign}.  It says that if $X$ is strongly regular, then the decomposition of $\mg$ into eigenspaces of $\ad_X^2$ is the same as the roots space decomposition of $\mg$ from Theorem~\ref{decompp}, with $\tau$ equal to the kernel of $\ad_X^2$.

If $X\in\tau$ is not strongly regular, then the eigenspace decomposition of $\ad_X^2$ is ``courser'' than the strongly regular one; that is, the root spaces $\ml_i$ for which $\alpha_i(X)=0$ are grouped with $\tau$ to form the kernel of $\ad_X^2$, and each other eigenspace is a root space or a sum of root spaces,
$\ml_{i_1}\oplus\cdots\oplus\ml_{i_k}$, coming from repeated values $\alpha_{i_1}(X)^2=\cdots=\alpha_{i_k}(X)^2$.  Since $\mg$'s decomposition into root spaces corresponds to the ``finest'' of the $\ad_X^2$ eigenspace decompositions, this decomposition is unique.  We have just established:
\begin{prop}\label{uniqq}
The decomposition from Theorem~\ref{decompp} is unique, and therefore $\{\pm\alpha_1,...,\pm\alpha_m\}$ are the only roots of $G$.
\end{prop}
\section{The bracket of two root spaces}
The roots describe exactly how vectors in $\tau$ bracket with vectors in the root spaces.  Surprisingly, they also help determine how vectors in one root space bracket with vectors in another root space.  If $\hat\alpha_i+\hat\alpha_j$ equals a dual root, then let $\ml_{ij}^+$ denote its root space; otherwise, let $\ml_{ij}^+:=\{0\}$.   If $\hat\alpha_i-\hat\alpha_j$ equals a dual root, then let $\ml_{ij}^-$ denote its root space; otherwise, let $\ml_{ij}^-:=\{0\}$.  With this notation:

\begin{theorem}\label{rootsums} $[\ml_i,\ml_j]\subset\ml_{ij}^+\oplus\ml_{ij}^-$.
\end{theorem}

In particular, if neither $\hat\alpha_i+\hat\alpha_j$ nor $\hat\alpha_i-\hat\alpha_j$ equals a dual root, then $[\ml_i,\ml_j]=\{0\}$.  For all of the classical groups except $SO(2n+1)$, we'll see that the sum and difference never both equal dual roots, so any pair of root spaces must bracket to zero or to a single root space:
\begin{cor} If $G\in\{SU(n),SO(2n),Sp(n)\}$, then for any pair $(i,j)$, either $[\ml_i,\ml_j]=0$ or there exists $k$ such that $[\ml_i,\ml_j]\subset\ml_k$.  In the latter case, $\hat\alpha_i\pm\hat\alpha_j = \pm\hat\alpha_k$.
\end{cor}
Section~10 will contain the standard proof of Theorem~\ref{rootsums} using complexified Lie algebras.  For now, we offer the following longer but less abstract proof:
\begin{proof}[Proof of Theorem~\ref{rootsums}]
Let $V\in\ml_i$ and let $W\in\ml_j$, and define $U:=[V,W]$.  We wish to prove that $U\in\ml_{ij}^+\oplus\ml_{ij}^-$.  First notice that $U\in\tau^\perp$ because for all $Y\in\tau$,
$$\lb U,Y\rb  = \lb[V,W],Y\rb = -\lb[V,Y],W\rb = \lb [Y,V],W\rb = 0.$$
For any $X\in\tau$, we can define the path $U_t:=[V_t,W_t]$, where
\begin{gather*}
V_t:=\Ad_{e^{tX}}V=\cos(\alpha_i(X)t)\cdot V + \sin(\alpha_i(X)t)\cdot R_i(V), \\
W_t:=\Ad_{e^{tX}}W=\cos(\alpha_j(X)t)\cdot W + \sin(\alpha_j(X)t)\cdot R_j(W)
\end{gather*}
are circles in $\ml_i$ and $\ml_j$. Standard trigonometric identities yield:
\begin{eqnarray}\label{tinker2}
2U_t & = & \,\,\,\,\,\cos((\alpha_i(X)+\alpha_j(X))t)([V,W]-[R_iV,R_jW])\notag\\
&   & + \sin((\alpha_i(X)+\alpha_j(X))t)([R_iV,W]+[V,R_jW])\\
&   & + \cos((\alpha_i(X)-\alpha_j(X))t)([V,W]+[R_iV,R_jW])\notag\\
&   & + \sin((\alpha_i(X)-\alpha_j(X))t)([R_iV,W]-[V,R_jW]).\notag
\end{eqnarray}
On the other hand, since $$U_t = [\Ad_{e^{tX}}V,\Ad_{e^{tX}}W]=\Ad_{e^{tX}}[V,W] = \Ad_{e^{tX}} U,$$ we can decompose
$U= U^1+\cdots+U^m$ (with $U^i\in\ml_i$), and Equation~\ref{RotateAdad} gives:
\begin{equation}\label{tinker1}U_t = \sum_{k=1}^m \left(\cos(\alpha_k(X)t)\cdot U^k +\sin(\alpha_k(X)t)\cdot R_k(U^k)\right).\end{equation}
The $k^\text{th}$ term of this sum, denoted $U_t^k$, is a circle in $\ml_k$.

For any $X\in\tau$, the expressions for $U_t$ obtained from Equations~\ref{tinker2} and~\ref{tinker1} must equal each other.  We claim this implies that the first two lines of Equation~\ref{tinker2} must form a circle in $\ml_{ij}^+$ and the last two lines must form a circle in $\ml_{ij}^-$, so in particular $U_t\in\ml_{ij}^+\oplus\ml_{ij}^-$ as desired.  This implication is perhaps most easily seen by considering special types of vectors $X\in\tau$, as follows.

First, choose $X\perp\text{span}\{\hat\alpha_i,\hat\alpha_j\}$, so $\alpha_i(X)=\alpha_j(X)=0$, so $t\mapsto U_t$ is constant.  If some $U^k\neq 0$, then $\alpha_k(X)=0$, which means that $\hat\alpha_k\perp X$.  In summary, if $U^k\neq 0$, then $\hat\alpha_k$ must be perpendicular to any $X$ which is perpendicular to $\text{span}\{\hat\alpha_i,\hat\alpha_j\}$.  We conclude that if $U^k\neq 0$, then $\hat\alpha_k\in\text{span}\{\hat\alpha_i,\hat\alpha_j\}$.

Next, choose $X\in\text{span}\{\hat\alpha_i,\hat\alpha_j\}$ with $\alpha_i(X)=\lb\hat\alpha_i,X\rb=1$ and $\alpha_j(X)=\lb\hat\alpha_j,X\rb=0$, which is possible because $\hat\alpha_i$ and $\hat\alpha_j$ are not parallel, by Lemma~\ref{notpar}.  If some $U^k\neq 0$, then Equation~\ref{tinker2} shows that $t\mapsto U_t^k$ has period $=2\pi$, so $\alpha_k(X) = \lb\hat\alpha_k,X\rb=\pm1$, by Equation~\ref{tinker1}.  In summary, if some $U^k\neq 0$, then $\hat\alpha_k\in\text{span}\{\hat\alpha_i,\hat\alpha_j\}$ has the same projection onto the orthogonal compliment of $\hat\alpha_j$ as does $\pm\hat\alpha_i$.  Reversing the roles of $i$ and $j$ shows that $\hat\alpha_k$ also has the same projection onto the orthogonal compliment of $\hat\alpha_i$ as does $\pm\hat\alpha_j$.  It follows easily that $\hat\alpha_k = \pm\hat\alpha_i\pm\hat\alpha_j$, so $U\in\ml_{ij}^+\oplus\ml_{ij}^-$.
\end{proof}

For a single $V\in\ml_i$ and $W\in\ml_j$, suppose we know the bracket $[V,W]=A^++A^-$ (with $A^+\in\ml_{ij}^+$ and $A^-\in\ml_{ij}^-$).  This single bracket determines the entire bracket operation between $\ml_i$ and $\ml_j$.  To see how, let $R_+$ denote the $90^\circ$ rotation of $\ml_{ij}^+$ which is counterclockwise if $\hat\alpha_i+\hat\alpha_j=\hat\alpha_k$ or clockwise if $\hat\alpha_i+\hat\alpha_j=-\hat\alpha_k$ for some $k$.  Similarly let $R_-$ denote the $90^\circ$ rotation of $\ml_{ij}^-$ which is counterclockwise if $\hat\alpha_i-\hat\alpha_j=\hat\alpha_k$ or clockwise if $\hat\alpha_i-\hat\alpha_j=-\hat\alpha_k$.   With this notation, the ideas of the previous proof yield the following generalization of Table~\ref{ijjkbrak}:

\begin{table}[h!]\begin{center}\begin{tabular}{ | c || c | c |} \hline
   $\mathbf{[\cdot,\cdot]}$ & $\mathbf{W}$ & $\mathbf{R_jW}$  \\ \hline\hline
   $\mathbf{V}$ & $A^++A^-$ & $R_+(A^+)-R_-(A^-)$  \\ \hline
   $\mathbf{R_iV}$ & $R_+(A^+)-R_-(A^-)$ & $-A^++A^-$ \\ \hline
\end{tabular}
\caption{The bracket $[\ml_{i},\ml_{j}]\subset\ml_{ij}^+\oplus\ml_{ij}^-$}\label{genijk}
\end{center}
\end{table}


\section{The structure of $\mg=so(2n)$}
Let $n>1$ and $G=SO(2n)$, so $\mg=so(2n)$.  Recall that the Lie algebra of the standard maximal torus of $G$ is:
$$\tau=\left\{\diag\left(\left(\begin{matrix} 0 & \theta_1 \\ -\theta_1 & 0\end{matrix}\right),...,\left(\begin{matrix} 0 & \theta_n \\ -\theta_n & 0\end{matrix}\right)\right)\mid \theta_i\in\R \right\}.$$
Let $H_i\in\tau$ denote the matrix with $\theta_i=1$ and all other $\theta$'s zero, so that $\{H_1,...,H_n\}$ is a basis for $\tau$. Also define:
$$E:=\left(\begin{matrix} 1 & 0 \\ 0 & 1\end{matrix}\right),F:=\left(\begin{matrix} 0 & 1 \\ -1 & 0\end{matrix}\right), X:=\left(\begin{matrix} 0 & 1 \\ 1 & 0\end{matrix}\right), Y:=\left(\begin{matrix} 1 & 0 \\ 0 & -1\end{matrix}\right).$$
Think of a matrix in $so(2n)$ as being an $n\times n$ grid of $2\times2$ blocks.  For each pair $(i,j)$ of distinct indices between $1$ and $n$, define $E_{ij}$ so that its $(i,j)^\text{th}$ block equals $E$ and its $(j,i)^\text{th}$ block equals -$E^T$ and all other blocks are zero.  Similarly define $F_{ij}$, $X_{ij}$ and $Y_{ij}$.  A basis of $\mg$ is formed from $\{H_1,...,H_n\}$ together with all $E$'s, $F$'s, $X$'s and $Y$'s with $i<j$.  In the case $n=2$, this basis looks like:

\scalebox{.7}{\parbox{\textwidth}{
\begin{align*}
H_{1} &= \left(\begin{matrix} 0 & 1 & 0 & 0\\ -1 & 0 & 0 & 0 \\ 0 & 0 & 0 & 0 \\ 0 & 0 & 0 & 0\end{matrix}\right),
E_{12} &= \left(\begin{matrix} 0 & 0 & 1 & 0\\ 0 & 0 & 0 & 1 \\ -1 & 0 & 0 & 0 \\ 0 & -1 & 0 & 0\end{matrix}\right),
X_{12} &= \left(\begin{matrix} 0 & 0 & 0 & 1\\ 0 & 0 & 1 & 0 \\ 0 & -1 & 0 & 0 \\ -1 & 0 & 0 & 0\end{matrix}\right),\\
H_{2} &= \left(\begin{matrix} 0 & 0 & 0 & 0\\ 0 & 0 & 0 & 0 \\ 0 & 0 & 0 & 1 \\ 0 & 0 & -1 & 0\end{matrix}\right),
F_{12} &= \left(\begin{matrix} 0 & 0 & 0 & 1\\ 0 & 0 & -1 & 0 \\ 0 & 1 & 0 & 0 \\ -1 & 0 & 0 & 0\end{matrix}\right),
Y_{12} &= \left(\begin{matrix} 0 & 0 & 1 & 0\\ 0 & 0 & 0 & -1 \\ -1 & 0 & 0 & 0 \\ 0 & 1 & 0 & 0\end{matrix}\right).
\end{align*}}}

Define $\ml_{ij}:=\text{span}\{E_{ij},F_{ij}\}$ and $\mathfrak{k}_{ij}:=\text{span}\{X_{ij},Y_{ij}\}$.  The decomposition of $\mg$ into root spaces is:
$$\mg=\tau\oplus\{\ml_{ij}\mid i<j\}\oplus\{\mk_{ij}\mid i<j\}.$$
Since $[E_{ij},F_{ij}] = 2(H_i-H_j)$ and $[X_{ij},Y_{ij}] = 2(H_i+H_j)$, the dual roots are the following matrices (and their negatives):
\begin{gather*}\hat\alpha_{ij}:=\left[\frac{E_{ij}}{|E_{ij}|},\frac{F_{ij}}{|F_{ij}|}\right] = \frac 12(H_i-H_j),\\ \hat\beta_{ij}:=\left[\frac{X_{ij}}{|X_{ij}|},\frac{Y_{ij}}{|Y_{ij}|}\right] = \frac 12(H_i+H_j).\end{gather*}
The following lists all sums and differences of dual roots which equal dual roots.  In each case, a sample bracket value is provided:
\begin{align*}
\hat\alpha_{ij}+\hat\alpha_{jk}&=\hat\alpha_{ik} & [\ml_{ij},\ml_{jk}]&\subset\ml_{ik} & [E_{ij},E_{jk}]&=E_{ik}\\
\hat\beta_{ij}-\hat\beta_{jk} &= \hat\alpha_{ik} & [\mk_{ij},\mk_{jk}]&\subset\ml_{ik} & [X_{ij},X_{jk}]&=E_{ik}\\
\hat\alpha_{ij}+\hat\beta_{jk} &= \hat\beta_{ik} & [\ml_{ij},\mk_{jk}]&\subset\mk_{ik} & [E_{ij},X_{jk}]&=X_{ik}
\end{align*}
The brackets of any pair of basis elements can be determined from the above sample bracket values via Table~\ref{genijk}, yielding:

\begin{table}[h!]\begin{center}
\begin{tabular}{ | c || c | c |} \hline
   $\mathbf{[\cdot,\cdot]}$ & $\mathbf{E_{jk}}$ & $\mathbf{F_{jk}}$  \\ \hline\hline
   $\mathbf{E_{ij}}$ & $E_{ik}$ & $F_{ik}$  \\ \hline
   $\mathbf{F_{ij}}$ & $F_{ik}$ & $-E_{ik}$ \\ \hline
\end{tabular}\,\,\,
\begin{tabular}{ | c || c | c |} \hline
   $\mathbf{[\cdot,\cdot]}$ & $\mathbf{X_{jk}}$ & $\mathbf{Y_{jk}}$  \\ \hline\hline
   $\mathbf{X_{ij}}$ & $E_{ik}$ & $-F_{ik}$  \\ \hline
   $\mathbf{Y_{ij}}$ & $F_{ik}$ & $E_{ik}$ \\ \hline
\end{tabular}\,\,\,
\begin{tabular}{ | c || c | c |} \hline
   $\mathbf{[\cdot,\cdot]}$ & $\mathbf{X_{jk}}$ & $\mathbf{Y_{jk}}$  \\ \hline\hline
   $\mathbf{E_{ij}}$ & $X_{ik}$ & $Y_{ik}$  \\ \hline
   $\mathbf{F_{ij}}$ & $Y_{ik}$ & $-X_{ik}$ \\ \hline
\end{tabular}\end{center}
\caption{The non-zero bracket relations for $\mg=so(2n)$}\label{so2nbraks}
\end{table}

\section{The structure of $\mg=so(2n+1)$}
Let $n>0$ and $G=SO(2n+1)$, so $\mg=so(2n+1)$.  Recall that the Lie algebra of the standard maximal torus of $G$ is:
$$\tau=\left\{\diag\left(\left(\begin{matrix} 0 & \theta_1 \\ -\theta_1 & 0\end{matrix}\right),...,\left(\begin{matrix} 0 & \theta_n \\ -\theta_n & 0\end{matrix}\right),0\right)\mid \theta_i\in\R \right\}.$$
Each of the previously defined elements of $so(2n)$ can be considered as an element of $so(2n+1)$ simply by adding a final row and final column of zeros.  In order to complete our previous basis of $so(2n)$ to a basis of $so(2n+1)$, we need the following additional matrices.  For each $1\leq i\leq n$, let $W_i\in so(2n+1)$ denote the matrix with entry $(2i-1,2n+1)$ equal to $1$ and entry $(2n+1,2i-1)$ equal to $-1$, and all other entries equal to zero.  Let $V_i$ denote the matrix with entry $(2i,2n+1)$ equal to $1$ and entry $(2n+1,2i)$ equal to $-1$, and all other entries equal to zero.  For $n=2$, these extra basis elements are:

\scalebox{.6}{\parbox{\textwidth}{
\begin{equation*}
W_{1} = \left(\begin{matrix} 0 & 0 & 0 & 0 & 1\\ 0 & 0 & 0 & 0 & 0 \\ 0 & 0 & 0 & 0 & 0\\ 0 & 0 & 0 & 0 & 0 \\ -1 & 0 & 0 & 0 & 0\end{matrix}\right),
V_{1} = \left(\begin{matrix} 0 & 0 & 0 & 0 & 0\\ 0 & 0 & 0 & 0 & 1 \\ 0 & 0 & 0 & 0 & 0\\ 0 & 0 & 0 & 0 & 0 \\ 0 & -1 & 0 & 0 & 0\end{matrix}\right),
W_{2} = \left(\begin{matrix} 0 & 0 & 0 & 0 & 0\\ 0 & 0 & 0 & 0 & 0 \\ 0 & 0 & 0 & 0 & 1\\ 0 & 0 & 0 & 0 & 0 \\ 0 & 0 & -1 & 0 & 0\end{matrix}\right),
V_{2} = \left(\begin{matrix} 0 & 0 & 0 & 0 & 0\\ 0 & 0 & 0 & 0 & 0 \\ 0 & 0 & 0 & 0 & 0\\ 0 & 0 & 0 & 0 & 1 \\ 0 & 0 & 0 & -1 & 0\end{matrix}\right).
\end{equation*}}}

Define $\ms_i:=\text{span}\{V_i,W_i\}$.  The root space decomposition is:
$$\mg=\tau\oplus\{\ml_{ij}\mid i<j\}\oplus\{\mk_{ij}\mid i<j\}\oplus\{\ms_i\mid 1\leq i\leq n\}.$$
Since $[V_i,W_i] = H_i$, the added dual roots are the following matrices (and their negatives):
$$\hat\gamma_{i}:=\left[\frac{V_i}{|V_i|},\frac{W_{i}}{|W_{i}|}\right] = \frac 12 H_i.$$
In addition to those of $so(2n)$, we have the following new sums and differences of dual roots which equal dual roots.  In each case, a sample bracket value is provided.
\begin{align*}
&\hat\alpha_{ij}-\hat\gamma_i=-\hat\gamma_j & &[\ml_{ij},\ms_i]\subset\ms_j & &[E_{ij},V_i]=-V_j \\
&\hat\alpha_{ij}+\hat\gamma_j=\hat\gamma_i  & &[\ml_{ij},\ms_j]\subset\ms_i & &[E_{ij},V_j]=V_i \\
&\hat\beta_{ij}-\hat\gamma_i=\hat\gamma_j   & &[\mk_{ij},\ms_i]\subset\ms_j & &[X_{ij},V_i]=-W_j \\
&\hat\beta_{ij}-\hat\gamma_j=\hat\gamma_i   & &[\mk_{ij},\ms_j]\subset\ms_i & &[X_{ij},V_j]=W_i \\
&\left.\begin{matrix}\hat\gamma_i-\hat\gamma_j =\hat\alpha_{ij} \\ \hat\gamma_i+\hat\gamma_j=\hat\beta_{ij}\end{matrix}\right\}
              & & [\ms_i,\ms_j]\subset\ml_{ij}\oplus\mh_{ij} &  &[V_i,V_j]=- \frac 12 E_{ij}+\frac 12 Y_{ij}
\end{align*}
Notice that $(\hat\gamma_i,\hat\gamma_j)$ is our first example of a pair of dual roots whose sum and difference \emph{both} equal dual roots.
Using Table~\ref{genijk}, the new bracket relations (in addition to those of Table~\ref{so2nbraks}) are summarized in Table~\ref{bracketodd}.

\begin{table}[h!]\begin{center}
\begin{tabular}{ | c || c | c |} \hline
   $\mathbf{[\cdot,\cdot]}$ & $\mathbf{V_i}$ & $\mathbf{W_i}$  \\ \hline\hline
   $\mathbf{E_{ij}}$ & $-V_j$ & $-W_j$  \\ \hline
   $\mathbf{F_{ij}}$ & $W_j$ & $-V_j$ \\ \hline
\end{tabular}\,\,\,
\begin{tabular}{ | c || c | c |} \hline
   $\mathbf{[\cdot,\cdot]}$ & $\mathbf{V_j}$ & $\mathbf{W_j}$  \\ \hline\hline
   $\mathbf{E_{ij}}$ & $V_i$ & $W_i$  \\ \hline
   $\mathbf{F_{ij}}$ & $W_i$ & $-V_i$ \\ \hline
\end{tabular}\,\,\,
\begin{tabular}{ | c || c | c |} \hline
   $\mathbf{[\cdot,\cdot]}$ & $\mathbf{V_i}$ & $\mathbf{W_i}$  \\ \hline\hline
   $\mathbf{X_{ij}}$ & $-W_j$ & $-V_j$  \\ \hline
   $\mathbf{Y_{ij}}$ & $V_j$ & $-W_j$ \\ \hline
\end{tabular}

\vspace{.1in}
\begin{tabular}{ | c || c | c |} \hline
   $\mathbf{[\cdot,\cdot]}$ & $\mathbf{V_j}$ & $\mathbf{W_j}$  \\ \hline\hline
   $\mathbf{X_{ij}}$ & $W_i$ & $V_i$  \\ \hline
   $\mathbf{Y_{ij}}$ & $-V_i$ & $W_i$ \\ \hline
\end{tabular}\,\,\,
\begin{tabular}{ | c || c | c |} \hline
   $\mathbf{[\cdot,\cdot]}$ & $\mathbf{V_j}$ & $\mathbf{W_j}$  \\ \hline\hline
   $\mathbf{V_i}$ & $-\frac 12 E_{ij} + \frac 12 Y_{ij}$ & $\frac 12 F_{ij} - \frac 12 X_{ij}$  \\ \hline
   $\mathbf{W_i}$ & $-\frac 12 F_{ij} - \frac 12 X_{ij}$ & $-\frac 12 E_{ij} - \frac 12 Y_{ij}$ \\ \hline
\end{tabular}
\end{center}
\caption{The additional bracket relations for $\mg=so(2n+1)$}\label{bracketodd}
\end{table}

\section{The structure of $\mg=sp(n)$}
Let $n>0$ and $G=Sp(n)$, so $\mg=sp(n)$.  Recall that the Lie algebra of the standard maximal torus of $G$ is:
$$\tau=\{\diag(\theta_1\ii,...,\theta_n\ii)\mid \theta_i\in\R\}.$$

For each index $i$, let $H_i$ denote the diagonal matrix with $\ii$ in position $(i,i)$ (and all other entries zero).  Let $J_i$ denote the diagonal matrix with $\jj$ in position $(i,i)$ (and all other entries zero), and let $K_i$ denote the diagonal matrix with $\kk$ in position $(i,i)$.  Notice that $H_i\in\tau$ and $J_i, K_i\in\tau^\perp$.

For each pair $(i,j)$ of distinct indices between $1$ and $n$, let $E_{ij}$ denote the matrix with $+1$ in position $(i,j)$ and $-1$ in position $(j,i)$.  Let $F_{ij}\in\mg$ denote the matrix with $\ii$ in positions $(i,j)$ and $(j,i)$.  Let $A_{ij}\in\mg$ denote the matrix with $\jj$ in positions $(i,j)$ and $(j,i)$.  Let $B_{ij}\in\mg$ denote the matrix with $\kk$ in positions $(i,j)$ and $(j,i)$.

The root space decomposition is:
\begin{eqnarray*}
sp(n) & = & \tau\oplus\{\text{span}\{E_{ij},F_{ij}\}\mid i<j\}\oplus\{\text{span}\{A_{ij},B_{ij}\}\mid i<j\}\\
      &   & \oplus\{\text{span}\{J_i,K_i\}\mid 1\leq i\leq n\}
\end{eqnarray*}
The dual roots are the following matrices (and their negatives):
\begin{gather*}
\hat\alpha_{ij}:=\left[\frac{E_{ij}}{|E_{ij}|},\frac{F_{ij}}{|F_{ij}|}\right] = H_i-H_j\\ \hat\beta_{ij}:=\left[\frac{A_{ij}}{|A_{ij}|},\frac{B_{ij}}{|B_{ij}|}\right] = H_i+H_j\\
\hat\gamma_{i}:=\left[\frac{J_{i}}{|J_{i}|},\frac{K_{i}}{|K_{i}|}\right] = 2H_i
\end{gather*}
Think of these dual roots initially as unrelated to the dual roots of $SO(2n+1)$ which bore the same names, but look for similarities.  The only sums or differences of dual roots which equal dual roots are listed below, with sample bracket values provided:
\begin{align*}
\hat\alpha_{ij}+\hat\alpha_{jk}&=\hat\alpha_{ik}  & [E_{ij},E_{jk}]&=2E_{ik}\\
\hat\beta_{ij}-\hat\beta_{jk} &= \hat\alpha_{ik}  & [A_{ij},A_{jk}]&=-2E_{ik}\\
\hat\alpha_{ij}+\hat\beta_{jk} &= \hat\beta_{ik}  & [E_{ij},A_{jk}]&=2A_{ik}\\
\hat\beta_{ij}-\hat\gamma_i & = -\hat\alpha_{ij}      & [A_{ij},J_i]&=2E_{ij}\\
\hat\beta_{ij}-\hat\gamma_j & = \hat\alpha_{ij}           & [A_{ij},J_j]&=-2E_{ij}\\
\hat\alpha_{ij}-\hat\gamma_i & = -\hat\beta_{ij}       & [A_{ij},J_i]&=-2A_{ij}\\
\hat\alpha_{ij} + \hat\gamma_j &= \hat\beta_{ij}         & [E_{ij},J_j]&=2A_{ij}
\end{align*}

We leave it to the reader in Exercise~\ref{spex} to list all non-zero brackets of pairs of basis vectors, using the above sample values together with Table~\ref{genijk}.
\section{The Weil Group}
Let $G$ be a compact Lie group with Lie algebra $\mg$.  Let $T\subset G$ be a maximal torus with Lie algebra $\tau\subset\mg$.  In this section, we will define and study the Weil group of $G$, which can be thought of as a group of symmetries of the roots of $G$.

First, let $N(T)$ denote the \underline{normalizer} of $T$, which means:
$$N(T):=\{g\in G\mid g T g^{-1}=T\}.$$
It is routine to check that $N(T)$ is a subgroup of $G$ and that $T$ is a normal subgroup of $N(T)$.

For each $g\in N(T)$, conjugation by $g$ is an automorphism of $T$, denoted $C_g:T\ra T$.  The derivative of $C_g$ at $I$ is the Lie algebra automorphism $\Ad_g:\tau\ra\tau$.  In fact, it is straightforward to see:
$$N(T) = \{g\in G\mid\Ad_g(\tau)=\tau\}.$$
One should expect automorphisms to preserve all of the fundamental structures of a Lie algebra, including its roots and dual roots.
\begin{prop}\label{permutedualroots}
For each $g\in N(T)$, $\Ad_g:\tau\ra\tau$ sends dual roots to dual roots.
\end{prop}
\begin{proof}
If $\hat\alpha\in\tau$ is a dual root with root space $\ml=\text{span}\{E,F\}$, then $\Ad_g\hat\alpha\in\tau$ is a dual root with root space $\Ad_g(\ml):=\text{span}\{\Ad_g E,\Ad_g F\}$.  This is because for all $X\in\tau$,
\begin{eqnarray*}
[X,\Ad_g E] & = & \Ad_g\left([\Ad_{g^{-1}}X,E]\right) = \Ad_g\left(\lb Ad_{g^{-1}} X,\hat\alpha\rb\cdot F \right) \\
            & = & \lb Ad_{g^{-1}} X,\hat\alpha\rb\cdot\Ad_g F = \lb X,\Ad_g\hat\alpha\rb\cdot\Ad_gF.
\end{eqnarray*}
Similarly, $[X,\Ad_gF] = -\lb X,\Ad_g\hat\alpha\rb\cdot \Ad_gE$, so $\Ad_g\hat\alpha$ is a dual-root according to Definition~\ref{D:dualroot}.
\end{proof}

Since the conjugates of $T$ cover $G$, $N(T)$ is not all of $G$.  In fact, we expect $N(T)$ to be quite small.  The following shows at least that $N(T)$ is larger than $T$.

\begin{prop}\label{hyper}
For each dual root, $\hat\alpha$, of $G$, there exists an element $g\in N(T)$ such that $\Ad_g(\hat\alpha)=-\hat\alpha$, and $\Ad_g(X)=X$ for all $X\in\tau$ with $X\perp\hat\alpha$.
\end{prop}
In other words, we can visualize $\Ad_g:\tau\ra\tau$ as a reflection through the ``hyperplane'' $\hat\alpha^\perp:=\{X\in\tau\mid X\perp\hat\alpha\}$.
\begin{proof}
Let $\hat\alpha$ be a dual root with root space $\ml=\text{span}\{E,F\}$.  Since $t\mapsto e^{tF}$ is a one-parameter group in $G$, $t\mapsto\Ad_{e^{tF}}$ is a one-parameter group of orthogonal automorphisms of $\mg$, with initial derivative equal to $\ad_F$.  Notice that:
\begin{gather*}
\text{For all }X\in\hat\alpha^\perp,\ad_F(X)=-[X,F] = \lb X,\hat\alpha\rb\cdot E = 0,  \\
\ad_F(E) =-[E,F] =-\hat\alpha = - |\hat\alpha|\cdot\frac{\hat\alpha}{|\hat\alpha|}, \\
\ad_F\left(\frac{\hat\alpha}{|\hat\alpha|}\right) = -\left[\frac{\hat\alpha}{|\hat\alpha|},F\right] = \left\langle\frac{\hat\alpha}{|\hat\alpha|},\hat\alpha\right\rangle \cdot E = |\hat\alpha|\cdot E.
\end{gather*}
Therefore $t\mapsto\Ad_{e^{tF}}=e^{\ad_{tF}}$ is a one-parameter group of orthogonal automorphisms of $\mg$ which acts as the identity on $\hat\alpha^\perp\subset\tau$ and which rotates $\text{span}\left\{\frac{\hat\alpha}{|\hat\alpha|},E\right\}$ with period $\frac{2\pi}{|\hat\alpha|}$.  Thus, at time $t_0:=\frac{\pi}{|\hat\alpha|}$, the rotation is half complete, so it sends $\hat\alpha\mapsto-\hat\alpha$.  Thus, the element $g=e^{t_0F}$ lies in $N(T)$ and acts on $\tau$ as claimed in the proposition.
\end{proof}

We would like to think of $N(T)$ as a group of orthogonal automorphisms of $\tau$, but the problem is that different elements of $N(T)$ may determine the same automorphism of $\tau$:
\begin{lem} For a pair $a,b\in N(T)$, $\Ad_a=\Ad_b$ on $\tau$ if and only if $a$ and $b$ lie in the same coset of $N(T)/T$.
\end{lem}
\begin{proof}
\begin{eqnarray*}
\Ad_{a}=\Ad_{b} \text { on }\tau & \Longleftrightarrow &C_{a}=C_{b} \text{ on } T\\
                                 & \Longleftrightarrow & C_{ab^{-1}}=\text{I} \text{ on } T \\
                                 & \Longleftrightarrow & ab^{-1}\text{ commutes with every element of }T\\
                                 & \Longleftrightarrow & ab^{-1}\in T
\end{eqnarray*}
\end{proof}

\begin{defn} The \underline{Weil group} of $G$ is $W(G):=N(T)/T$.
\end{defn}

So it is not $N(T)$ but $W(G)$ which should be thought of as a group of orthogonal automorphisms of $\tau$.  Each $w=g\cdot T\in W(G)$ determines the automorphism of $\tau$ which sends $X\in\tau$ to $$w\star X:=\Ad_g X.$$  By the previous Lemma, $w\star X$ is well-defined (independent of the coset representative $g\in N(T)$), and different elements of $W(G)$ determine different automorphisms of $\tau$.


\begin{prop} $W(G)$ is finite.
\end{prop}
\begin{proof}
By the above remarks, $W(G)$ is isomorphic to a subgroup of the group of automorphisms of $\tau$.  By Proposition~\ref{permutedualroots}, each $w\in W(G)$ determines a permutation of the $2m$ dual roots of $G$.  If two elements $w_1,w_2\in W$ determine the same permutation of the dual roots, then they determine the same linear map on the span of the dual roots.  The proof of Proposition~\ref{rootsspan} shows that the span of the dual roots equals the orthogonal compliment in $\tau$ of the Lie algebra of the center of $G$.  Since each element of $W(G)$ acts as the identity on Lie algebra of the center of $G$, this shows that $w_1$ and $w_2$ determine the same automorphism of $\tau$, and therefore $w_1=w_2$.  Thus, different elements of $W(G)$ must determine different permutations of the dual roots.  It follows that $W(G)$ is isomorphic to a subgroup of the group of permutations of the $2m$ dual roots, and thus has finite order which divides $(2m)!$
\end{proof}

Proposition~\ref{hyper} guarantees that for each dual root $\hat\alpha$, there exists an element $w_{\hat\alpha}\in W(G)$ such that $w_{\hat\alpha}\star\hat\alpha = -\hat\alpha$ and $w_{\hat\alpha}\star X=X$ for all $X\in\hat\alpha^\perp$.  It turns out that such elements generate $W(G)$:
\begin{prop} \label{BD} Every element of $W(G)$ equals a product of finitely many of the $w_{\hat\alpha}$'s.
\end{prop}

We will not prove this proposition.  It implies that $W(G)$ depends only on the Lie algebra.  That is, if two Lie groups have isomorphic Lie algebras, then they have isomorphic Weil groups.  By contrast, the normalizer of the maximal torus of $SO(3)$ is not isomorphic to that of $Sp(1)$, even though $so(3)\cong sp(1)$ (see Exercises 9.12 and 9.13 for descriptions of these normalizers).

It is useful to derive an explicit formula for the reflection through the hyperplane $\hat\alpha^\perp$:
\begin{lem}\label{hyperproj}
If $\hat\alpha$ is a dual root, then $w_{\hat\alpha}\star X = X-2\frac{\lb\hat\alpha,X\rb}{\lb\hat\alpha,\hat\alpha\rb}\hat\alpha$ for all $X\in\tau$.
\end{lem}
\begin{proof}
$X$ uniquely decomposes as the sum of a vector parallel to $\hat\alpha$ and a vector perpendicular to $\hat\alpha$ in the following explicit manner:
$$X=X_\parallel+X_\perp = \left(\frac{\lb\hat\alpha,X\rb}{\lb\hat\alpha,\hat\alpha\rb}\hat\alpha\right) + \left(X-\frac{\lb\hat\alpha,X\rb}{\lb\hat\alpha,\hat\alpha\rb}\hat\alpha\right).$$
We have $w_{\hat\alpha}(X) = -X_\parallel + X_\perp = X-2\frac{\lb\hat\alpha,X\rb}{\lb\hat\alpha,\hat\alpha\rb}\hat\alpha$.
\end{proof}
\begin{prop} $W(SU(n))$ is isomorphic to $S_{n}$, the group of all permutations of $n$ objects.
\end{prop}
\begin{proof}
Using Lemma~\ref{hyperproj}, one can check that $w_{\hat\alpha_{ij}}\star X$ is obtained from $X\in\tau$ by exchanging the $i^\text{th}$ and $j^\text{th}$ diagonal entries.  For example,
$$w_{\hat\alpha_{12}}\star\diag(\lambda_1\ii,\lambda_2\ii,\lambda_3\ii,...,\lambda_n\ii) = \diag(\lambda_2\ii,\lambda_1\ii,\lambda_3\ii,...,\lambda_n\ii).$$
The collection $\{w_{\hat\alpha_{ij}}\}$ generates the group, $S_n$,  of all permutations of the $n$ diagonal entries, so Proposition~\ref{BD} implies that $W(SU(n))$ is isomorphic to $S_n$.
\end{proof}

An explicit coset representatives, $g_{ij}\in N(T)\subset SU(n)$, for each $w_{\hat\alpha_{ij}}$ can be found using the construction in the proof of Proposition~\ref{hyper}.  For example, in $G=SU(3)$, we can choose:
$$g_{23}=e^{(\pi/2)E_{23}} = \left(\begin{matrix} 1 & 0 & 0 \\ 0 & 0 & 1  \\ 0 & -1 & 0 \end{matrix}\right),
\text{ or }
g_{23}=e^{(\pi/2)F_{23}} = \left(\begin{matrix} 1 & 0 & 0  \\ 0 & 0 & \ii  \\ 0 & \ii & 0 \end{matrix}\right).
$$

Finally, we will determine the Weil groups of the remaining classical groups.  For each of $G\in\{SO(2n),SO(2n+1),Sp(n)\}$, we previously chose a basis of $\tau$, which in all three cases was denoted $\{H_1,...,H_n\}$.  These basis elements are mutually orthogonal and have the same length, $l$.  They are tangent to the circles which comprise $T$, so they generate one-parameter groups, $t\mapsto e^{t H_i}$, with period $2\pi$.  In fact, $\{\pm H_1,...,\pm H_n\}$ are the only vectors in $\tau$ of length $l$ which generate one-parameter groups with period $2\pi$.  For any $g\in N(T)$, $\Ad_g:\tau\ra\tau$ must preserve this property and therefore must permute the set $\{\pm H_1,...,\pm H_n\}$.  We will think of $W(G)$ as a group of permutations of this set (rather than of the set of dual roots).

\begin{prop} $|W(SO(2n+1))|=|W(Sp(n))| = 2^nn!$, and $|W(SO(2n))|=2^{n-1}n!$
\end{prop}
\begin{proof}
For each of $G\in\{SO(2n),SO(2n+1),Sp(n)\}$, there are dual roots denoted $\hat\alpha_{ij}$ and $\hat\beta_{ij}$.  Using Proposition~\ref{hyperproj}, the corresponding Weil group elements permute the set $\{\pm H_1,...,\pm H_n\}$ as follows:
\begin{gather*}
w_{\hat\alpha_{ij}} \text{ sends }H_i\mapsto H_j,\,\,\,\, H_j\mapsto H_i,\,\,\,\, H_k\mapsto H_k\text{ for all }k\notin\{i,j\},\\
w_{\hat\beta_{ij}}\text{ sends }H_i\mapsto -H_j,\,\,\,\, H_j\mapsto-H_i,\,\,\,\,  H_k\mapsto H_k\text{ for all }k\notin\{i,j\}.
\end{gather*}
For $G\in\{SO(2n+1),Sp(n)\}$ we additionally have dual roots denoted $\{\hat\gamma_i\}$ which give the following permutations:
$$w_{\hat\gamma_i}\text{ sends }H_i\mapsto -H_i,\,\,\,\,H_k\mapsto H_k\text{ for all }k\neq i.$$
By Proposition~\ref{BD}, the Weil group is isomorphic to the group of permutations of set $\{\pm H_1,...,\pm H_n\}$ generated by the above permutations.  For $G\in\{SO(2n+1),Sp(n)\}$, one can generate any permutation of the $n$ indices together with any designation of which of the $n$ indices become negative, giving $2^nn!$ possibilities.  For $G=SO(2n)$, the number of negative indices must be even (check that this is the only restriction), so there are half as many total possibilities.
\end{proof}
\section{Towards the classification theorem}
In this section, we very roughly indicate the proof of the previously-mentioned classification theorem for compact Lie groups, which stated:
\begin{theorem}\label{classifyagain}The Lie algebra of every compact Lie group, $G$, is isomorphic to the Lie algebra of a product $G_1\times G_2\times\cdots\times G_k$, where each $G_i$ is one of $\{SO(n), SU(n), Sp(n)\}$ for some $n$, or is one of the five exceptional Lie groups: $G_2, F_4, E_6, E_7$ and $E_8$.
\end{theorem}

It suffices to prove this theorem assuming that $G$ has a finite center, so we'll henceforth assume that all of our compact Lie groups have finite centers.  In this case, the dual roots of $G$ are a finite collection of vectors in $\tau$ which form a ``root system'' according to the following definition:
\begin{defn} Let $\tau$ be a real vector space which has an inner product, $\lb\cdot,\cdot\rb$.  Let $R$ be a finite collection of nonzero vectors in $\tau$ which spans $\tau$.  The pair $(\tau,R)$ is called a \underline{root system} if the following properties are satisfied:
\begin{enumerate}
\item If $\alpha\in R$, then $-\alpha\in R$, but no other multiple of $\alpha$ is in $R$.
\item If $\alpha,\beta\in R$, then $w_\alpha\star\beta:=\beta-2\frac{\lb\beta,\alpha\rb}{\lb\alpha,\alpha\rb}\alpha\in R$.
\item If $\alpha,\beta\in R$, then the quantity $2\frac{\lb\beta,\alpha\rb}{\lb\alpha,\alpha\rb}$ is an integer.
\end{enumerate}
In this case, the elements of $R$ are called \underline{roots}, and the dimension of $\tau$ is called the \underline{rank} of the root system.
\end{defn}

In property (2), each $a\in\tau$ determines the orthogonal endomorphism of $\tau$ which sends $X\in\tau$ to the vector $w_a\star X\in\tau$ defined as $w_a\star X := X-2\frac{\lb X,a\rb}{\lb a,a\rb}a$.  That is, $w_a:\tau\ra\tau$ is the reflection through the hyperplane $a^\perp$.  Property (2) says that for each root $\alpha$, the reflection $w_\alpha$ sends roots to roots.  The \underline{Weil group} of $(\tau,R)$, denoted $W(\tau,R)$, is defined as the group of all endomorphisms of $\tau$ obtained by composing a finite number of the $w_\alpha$'s.  As before, $W(\tau,R)$ is isomorphic to a subgroup of the group of permutations of $R$.

Some representation theory is required to prove that the dual roots of $G$ satisfy property (3).  Interpreting $\frac{\lb\beta,\alpha\rb}{\lb\alpha,\alpha\rb}$ as in the proof of Lemma~\ref{hyperproj}, property (3) says that the projection of $\beta$ onto $\alpha$ (previously denoted $\beta_\parallel$) must be an integer or half-integer multiple of $\alpha$, and vice-versa.  This implies very strong restrictions on the angle $\angle(\alpha,\beta)$ and on the ratio $\frac{|\alpha|}{|\beta|}$.  In particular, the following is straightforward to prove using only property (3):
\begin{prop}\label{anglechoices}
Let $(\tau,R)$ be a root system.  If $\alpha,\beta\in R$, then one of the following holds:
\begin{enumerate}
\item[(0)] $\lb \alpha,\beta\rb =0$.
\item[(1)] $|\alpha|=|\beta|$ and $\angle(\alpha,\beta)\in\{60^\circ,120^\circ\}$.
\item[(2)] $\max\{|\alpha|,|\beta|\}=\sqrt{2}\cdot\min\{|\alpha|,|\beta|\}$ and $\angle(\alpha,\beta)\in\{45^\circ,135^\circ\}$.
\item[(3)] $\max\{|\alpha|,|\beta|\}=\sqrt{3}\cdot\min\{|\alpha|,|\beta|\}$ and $\angle(\alpha,\beta)\in\{30^\circ,150^\circ\}$.
\end{enumerate}
\end{prop}

In fact, the definition of a root system is so restrictive, root systems have been completely classified:
\begin{theorem}\label{hard1}Every root system is equivalent to the system of dual roots for a Lie group of the form $G=G_1\times G_2\times\cdots\times G_k$, where each $G_i$ is one of $\{SO(n), SU(n), Sp(n)\}$ for some $n$, or is one of the five exceptional Lie groups.
\end{theorem}
In order for this classification of root systems to yield a proof of Theorem~\ref{classifyagain}, it remains only to establish that:
\begin{theorem}\label{hard2}
Two compact Lie groups with equivalent systems of dual roots must have isomorphic Lie algebras.
\end{theorem}
The notion of equivalence in the previous two theorems is formalized as follows:
\begin{defn}
The root system $(\tau,R)$ is said to be \underline{equivalent} to the root system $(\tau',R')$ if there exists a linear isomorphism $f:\tau\ra\tau'$ which sends $R$ onto $R'$ such that for all $\alpha\in R$ and $X\in\tau$ we have:
$$f(w_\alpha\star X) = w_{f(\alpha)}\star f(X).$$
\end{defn}
If $f$ is orthogonal (meaning that $\lb f(X),f(Y)\rb = \lb X,Y\rb$ for all $X,Y\in\tau$), then the hyperplane-reflection property in this definition is automatic.  An example of a non-orthogonal equivalence is given in Exercise~\ref{subrootsystem}.

The proofs of Theorems~\ref{hard1} and~\ref{hard2} are difficult; see~[9] for complete details. One of the key steps in the proof of Theorem~\ref{hard1} involves showing that every root system contains a special type of basis called a ``base,'' defined as follows:

\begin{defn} Let $(\tau,R)$ be a root system, and let $\Delta\subset R$ be a collection of the roots which forms a basis of $\tau$, which implies that every $\alpha\in R$ can be written uniquely as a linear combination of elements of $\Delta$.  We call $\Delta$ a \underline{base} of $R$ if the non-zero coefficients in each such linear combination are integers and are either all positive (in which case $\alpha$ is called a \underline{positive root}) or all negative (in which case $\alpha$ is called a \underline{negative root}).
\end{defn}

For a proof that every root system has a base, see~\cite{Hall} or~\cite{Hel}.  The most natural base for the system of dual roots of $G=SU(n)$ is: $$\Delta=\{\hat\alpha_{12},\hat\alpha_{23},...,\hat\alpha_{(n-1)n}\}.$$  This base induces the same notion of ``positive'' that was provided in Section~1; namely, $\hat\alpha_{ij}$ is positive if and only if $i<j$.  For example, $\hat\alpha_{25}$ is positive because $\hat\alpha_{25}=\hat\alpha_{23}+\hat\alpha_{34}+\hat\alpha_{45}$.

A natural base for the system of dual roots of $G=SO(2n)$ is:
$$\Delta=\{\hat\alpha_{12},\hat\alpha_{23},...,\hat\alpha_{(n-1)n},\hat\beta_{(n-1)n}\}.$$
As before, $\hat\alpha_{ij}$ is positive if and only if $i<j$.  Also, $\hat\beta_{ij}$ is positive and $-\hat\beta_{ij}$ is negative for each pair $(i,j)$.  For example, when $n=8$, so $G=SO(16)$, we can verify that $\hat\beta_{35}$ is positive by writing:
$$\hat\beta_{35}=\hat\alpha_{34}+\hat\alpha_{45}+2\hat\alpha_{56}+2\hat\alpha_{67}+\hat\alpha_{78}+\hat\beta_{78}.$$

A base for the system of dual roots of $G=SO(2n+1)$ is:
$$\Delta=\{\hat\alpha_{12},\hat\alpha_{23},...,\hat\alpha_{(n-1)n},\hat\gamma_n\}.$$
As before, $\hat\alpha_{ij}$ is positive if and only if $i<j$, each $\hat\beta_{ij}$ is positive, and each $-\hat\beta_{ij}$ is negative.  Further, $\hat\gamma_i$ is positive and $-\hat\gamma_i$ is negative for each index $1\leq i\leq n$.  For example, when $n=8$, so $G=SO(17)$, we can verify that $\hat\beta_{35}$ and $\hat\gamma_3$ are positive by writing:
\begin{gather*}\hat\beta_{35}=\hat\alpha_{34}+\hat\alpha_{45}+2\hat\alpha_{56}+2\hat\alpha_{67}+2\hat\alpha_{78}+2\hat\gamma_{8},\\
\hat\gamma_3 = \hat\alpha_{34}+\hat\alpha_{45}+\hat\alpha_{56}+\hat\alpha_{67}+\hat\alpha_{78}+\hat\gamma_{8}.\end{gather*}

A base for the system of dual roots of $G=Sp(n)$ is:
$$\Delta=\{\hat\alpha_{12},\hat\alpha_{23},...,\hat\alpha_{(n-1)n},\hat\gamma_n\}.$$
As before, $\hat\alpha_{ij}$ is positive if and only if $i<j$, each $\hat\beta_{ij}$ and each $\hat\gamma_i$ is positive, and each $-\hat\beta_{ij}$ and each $-\hat\gamma_i$ is negative.  For example, in $G=Sp(8)$, we can verify that $\hat\beta_{35}$ and $\hat\gamma_3$ are positive by writing:
\begin{gather*}\hat\beta_{35}=\hat\alpha_{34}+\hat\alpha_{45}+2\hat\alpha_{56}+2\hat\alpha_{67}+2\hat\alpha_{78}+\hat\gamma_{8},\\
\hat\gamma_3 = 2\hat\alpha_{34}+2\hat\alpha_{45}+2\hat\alpha_{56}+2\hat\alpha_{67}+2\hat\alpha_{78}+\hat\gamma_{8}.\end{gather*}

For a compact Lie group, $G$, each root space, $\ml$, is associated with two dual roots.  A base, $\Delta$, will designate one of them as positive (denoted $\hat\alpha$) and the other as negative (denoted $-\hat\alpha$).  Therefore, a base provides a notion of ``clockwise'' for each $\ml$; namely, clockwise with respect to an ordered orthonormal basis $\{E,F\}$ of $\ml$ such that $[E,F]$ equals the positive dual root.  None of the above bases for the classical groups are unique, which reflects the lack of a canonical notion of clockwise for the individual root spaces.  A different base would induce a different division of the roots into positive and negative roots, and thus different notions of clockwise for the root spaces.

\begin{lem}\label{lobtuse} If $(\tau,R)$ is a root system, $\Delta$ is a base, and $\alpha,\beta\in\Delta$, then one of the following holds:
\begin{enumerate}
\item[(0)] $\lb \alpha,\beta\rb = 0$.
\item[(1)] $|\alpha|=|\beta|$ and $\angle(\alpha,\beta)=120^\circ$.
\item[(2)] $\text{max}\{|\alpha|,|\beta|\}=\sqrt{2}\cdot\text{min}\{|\alpha|,|\beta|\}$, and $\angle(\alpha,\beta)=135^\circ$.
\item[(3)] $\text{max}\{|\alpha|,|\beta|\}=\sqrt{3}\cdot\text{min}\{|\alpha|,|\beta|\}$, and $\angle(\alpha,\beta)=150^\circ$.
\end{enumerate}
\end{lem}

\begin{proof} By Proposition~\ref{anglechoices}, we need only prove that $\angle(\alpha,\beta)$ is not acute.  For each of the three possible acute angles, it is straightforward to show that either $w_{\alpha}\star\beta = \alpha-\beta$ or $w_{\beta}\star\alpha = \beta-\alpha$.  In either case, $\alpha-\beta\in R$ is a root whose unique expression as a linear combination of elements from $\Delta$ has a positive and a negative coefficient, contradicting the definition of base.
\end{proof}

It turns out that to determine the equivalence class of a root system $(\tau,R)$, one only needs to know the angles between pairs of vectors from a base, $\Delta$, of the root system.  A \underline{Dynkin diagram} is a graph which encodes exactly this information.  The nodes of the Dynkin diagram are the elements of $\Delta$ (so the number of nodes equals the rank of the root system).  For a pair of nodes representing elements $\alpha,\beta\in\Delta$, we put 0, 1, 2, or 3 edges between them to represent the possibilities enumerated in Lemma~\ref{lobtuse}.  Further, we decorate each double or triple edge with an arrow from the vertex associated with the longer root towards the vertex associated with the smaller root.
It can be proven that the Dynkin diagram does not depend on the choice of base, and that it determines the equivalence class of the root system.  The classification of root systems was achieved by classifying all possible Dynkin diagrams. The Dynkin diagrams for systems of dual roots of the classical groups are pictured in Figure~\ref{F:dynkin}.
\begin{figure}[h!]
   \scalebox{.35}{\includegraphics{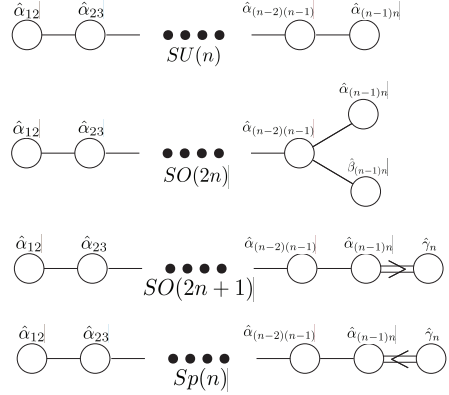}}
   \caption{The Dynkin diagrams of the classical Lie groups.}\label{F:dynkin}
   \end{figure}
\section{Complexified Lie algebras}
In this section, we will define the ``complexification'' of a Lie algebra, to build a bridge between this book and more advanced books which typically emphasize roots of a complexified Lie algebra.
\begin{defn} Let $V$ be an $n$-dimensional vector space over $\R$.  The \underline{complexification} of $V$ is defined as:
$$V_\C:=\{X+ Y\ii\mid X,Y\in V\}.$$
\end{defn}

Notice that $V_\C$ is an $n$-dimensional vector space over $\C$, with vector addition and scalar multiplication  defined in the obvious way:
\begin{align*}
(X_1+ Y_1\ii)+(X_2+ Y_2\ii)&:=(X_1+X_2)+(Y_1+Y_2)\ii, \\
(a+b\ii)\cdot(X+Y\ii)&:=(a\cdot X-b\cdot Y)+(b\cdot X+a\cdot Y)\ii,\end{align*}
for all $X,X_1,X_2,Y,Y_1,Y_2\in V$ and $a,b\in\R$.

If $V$ has an inner product, $\lb\cdot,\cdot\rb$, then this induces a natural complex-valued inner product on $V_\C$ defined as:
$$\lb X_1+Y_1\ii,X_2+Y_2\ii\rb_\C:=(\lb X_1,X_2\rb +\lb Y_1,Y_2\rb)+(\lb Y_1,X_2\rb -\lb X_1,Y_2\rb)\ii,$$
which is designed to satisfy all of the familiar properties of the standard hermitian inner product on $\C^n$ enumerated in Prop. 3.3.

If $G$ is a Lie group with Lie algebra $\mg$, then $\mg_\C$ inherits a ``complex Lie bracket'' operation defined in the most natural way:
$$[X_1+Y_1\ii,X_2+Y_2\ii]_\C:=([X_1,X_2]-[Y_1,Y_2])+([X_1,Y_2]+[Y_1,X_2])\ii.$$
This operation satisfies the familiar Lie bracket properties from Proposition 8.4 (with scalars $\lambda_1,\lambda_2\in\C$), including the Jacobi identity.

A potential confusion arises when $G\subset GL(n,\C)$ or $GL(n,\HH)$, since the symbol ``$\ii$'' already has a meaning for the entries of matrices in $\mg$.  In these cases, one should choose a different (initially unrelated) symbol, like ``$\mathbf{I}$'', for denoting elements of $\mg_\C$.  See Exercises~\ref{complexalgebras1} and~\ref{complexalgebras2} for descriptions of $so(n)_\C$ and $u(n)_\C$.

If $G$ is a compact Lie group, then the root space decomposition, $\mg=\tau\oplus\ml_1\oplus\cdots\oplus\ml_m$, induces a decomposition of $\mg_\C$ which is orthogonal with respect to $\lb\cdot,\cdot\rb_\C$:
$$\mg_\C = \tau_\C\oplus(\ml_1)_\C\oplus\cdots\oplus(\ml_m)_\C.$$
Notice that $\tau_\C$ is an abelian $\C$-subspace of $\mg_\C$ (``abelian'' means that every pair of vectors in $\tau_\C$ brackets to zero), and is maximal in the sense that it is not contained in any larger abelian $\C$-subspace of $\mg_\C$.

Each root space $\ml_i$ is associated with two roots, called $\alpha_i$ and $-\alpha_i$.  We intend to further decompose each $(\ml_i)_\C$ into two 1-dimensional $\C$-subspaces, one for each of these two roots.  That is, we will write:
\begin{equation}\label{splitagain}(\ml_i)_\C=\mg_{\overline{\alpha}_i}\oplus\mg_{\overline{-\alpha}_i},\end{equation}
with this notation defined as follows:
\begin{defn}\label{tyit}
If $\alpha$ is a root of $G$, and $\{E,F\}$ is an orthonormal basis of the corresponding root space, $\ml$, ordered so that the corresponding dual root is $\hat\alpha=[E,F]$, then:
\begin{enumerate}
\item Define the $\C$-linear function $\overline\alpha:\tau_\C\ra\C$ so that for all $X=X_1+X_2\ii\in\tau_\C$, we have $$\overline{\alpha}(X) := (-\ii)\cdot(\alpha(X_1)+\alpha(X_2)\ii) = \alpha(X_2)-\alpha(X_1)\ii.$$
\item Define $\mg_{\overline\alpha}:=\text{span}_\C\{E+F\ii\}=\{\lambda(E+F\ii)\mid\lambda\in\C\}\subset\ml_\C.$
\end{enumerate}
\end{defn}

If $\{E,F\}$ is a correctly-ordered basis for $\alpha$, then one for $-\alpha$ is $\{F,E\}$ or $\{E,-F\}$.  In Equation~\ref{splitagain}, notice that $\mg_{\overline\alpha}=\text{span}_\C\{E+F\ii\}$ and $\mg_{\overline{-\alpha}}=\text{span}_\C\{F+E\ii\}$ are orthogonal with respect to $\lb\cdot,\cdot\rb_\C$.

The space $\mg_{\overline\alpha}$ is well-defined, meaning independent of the choice of basis $\{E,F\}$.  To see this, notice that another correctly-ordered basis would look like $\{R_\theta E,R_\theta F\}$, where $R_\theta$ denotes a counterclockwise rotation of $\ml$ through angle $\theta$ (this assertion is justified in Exercise~\ref{bellpeppers}).  Setting $\lambda=e^{-i\theta}=\cos\theta-\ii\sin\theta$ gives:
\begin{eqnarray*}
\lambda\cdot(E+F\ii) & = & ((\cos\theta) E+(\sin\theta) F)+((\cos\theta) F-(\sin\theta) E)\ii \\
                     & = & (R_\theta E) + (R_\theta F)\ii.
\end{eqnarray*}
Thus, $\text{span}_\C\{E+F\ii\}=\text{span}_\C\{(R_\theta E)+(R_\theta F)\ii\}$.

The motivation for Definition~\ref{tyit} is the following:
\begin{prop} If $\alpha$ is a root of $G$, then for each $X\in\tau_\C$, the value $\overline\alpha(X)\in\C$ is an eigenvalue of the function $\ad_X:\mg_\C\ra\mg_\C$ (which sends $V\mapsto[X,V]_\C$), and each vector in $\mg_{\overline\alpha}$ is a corresponding eigenvector.
\end{prop}
\begin{proof}
By $\C$-linearity, it suffices to verify this for $X\in\tau$, which is done as follows:
\begin{eqnarray*}
[X,E+F\ii]_\C & = & [X,E]+[X,F]\ii = \alpha(X) (F-E\ii)\\
                  & = &  (-\ii\cdot\alpha(X))(E+F\ii)=\overline\alpha(X)\cdot(E+F\ii).
\end{eqnarray*}
\end{proof}

For our general definition, we will use:
\begin{defn}
A non-zero $\C$-linear function $\omega:\tau_\C\ra\C$ is called a \underline{complex root} of $\mg_\C$ if there exists a non-zero $\C$-subspace $\mg_\omega\subset\mg_\C$ (called a \underline{complex root space}) such that for all $X\in\tau_\C$ and all $V\in\mg_\omega$ we have:
$$[X,V]_\C=\omega(X)\cdot V.$$
In other words, the elements of $\mg_\omega$ are eigenvectors of $\ad_X$ for each $X\in\tau_\C$, and $\omega$ catalogs the corresponding eigenvalues.
\end{defn}
The notations ``$\mg_\omega$'' and ``$\mg_{\overline{\alpha}}$'' are consistent because of:

\begin{prop} If $\alpha$ is a root of $G$, then $\overline{\alpha}$ is a complex root of $\mg_\C$, and all complex roots of $\mg_\C$ come from roots of $G$ in this way.  Thus, $\mg_\C$ decomposes uniquely as an orthogonal direct sum of complex root spaces:
\begin{eqnarray*}
\mg_\C  & = & \tau_\C\oplus\{\mg_\omega\mid\omega\text{ is a complex root of }\mg_\C\}\\
        & = & \tau_\C\oplus\{\mg_{\overline{\alpha}}\mid\alpha\text{ is a root of }G\}\\
        & = & \tau_\C\oplus\mg_{\overline{\alpha}_1}\oplus\mg_{-\overline\alpha_1}\oplus\cdots
                     \oplus\mg_{\overline{\alpha}_m}\oplus\mg_{-\overline\alpha_m}.
\end{eqnarray*}
\end{prop}

One advantage of the complex setting is that $\overline\alpha_i$ and $-\overline\alpha_i$ correspond to different complex root spaces, so complex roots correspond one-to-one with complex root spaces.  Another advantage is that the following complexified version of Theorem~\ref{rootsums} has a short proof:

\begin{lem}\label{complexrootsums}
Suppose that $\omega_1$ and $\omega_2$ are complex roots of $\mg_\C$.  If $V_1\in\mg_{\omega_1}$ and $V_2\in\mg_{\omega_2}$, then
$$[V_1,V_2]_\C\in\begin{cases} \tau &\text{ if }\omega_1=-\omega_2\\ \mg_{\omega_1+\omega_2}&\text{ if } \omega_1+\omega_2\text{ is a complex root of }\mg_\C \\ \{0\}&\text{ otherwise.}\end{cases}$$
\end{lem}
\begin{proof}
Omitting the ``$\C$'' subscripts of Lie brackets for clarity, the complexified version of the Jacobi identity gives that for all $X\in\tau_\C$:
\begin{eqnarray*}
[X,[V_1,V_2]] & = & -[V_1,[V_2,X]]-[V_2,[X,V_1]]\\
              & = & -[[X,V_2],V_1]+[[X,V_1],V_2]]\\
              & = & (\omega_1(X)+\omega_2(X))[V_1,V_2],
\end{eqnarray*}
from which the three cases follow.
\end{proof}

\begin{proof}[Alternative proof of Theorem~\ref{rootsums}]
For distinct indices $i,j$,
$$\mg_{\pm\overline\alpha_i}=\text{span}_\C\{E_i\pm F_i\ii\}\,\text{ and }\,\,\mg_{\pm\overline\alpha_j}=\text{span}_\C\{E_j\pm F_j\ii\}.$$
Lemma~\ref{complexrootsums} says that following two brackets
\begin{gather*}
[E_i+F_i\ii,E_j+F_j\ii]_\C = ([E_i,E_j]-[F_i,F_j])+([E_i,F_j]+[F_i,E_j])\ii,\\
[E_i+F_i\ii,E_j-F_j\ii]_\C = ([E_i,E_j]+[F_i,F_j])+(-[E_i,F_j]+[F_i,E_j])\ii
\end{gather*}
lie respectively in $\mg_{\overline\alpha_i+\overline\alpha_j}$ and $\mg_{\overline\alpha_i-\overline\alpha_j}$.  The convention here is that $\mg_{\omega}:=\{0\}$ if $\omega$ is not a complex root.  Write $\ml_{ij}^+=\text{span}\{E_{ij}^+,F_{ij}^+\}$ and $\ml_{ij}^-=\text{span}\{E_{ij}^-,F_{ij}^-\}$, where these basis vectors may be zero.  For the sum of the above two vectors, we have:
\begin{eqnarray*}
2[E_i,E_j] + 2[F_i,E_j]\ii & \in & \mg_{\overline\alpha_i+\overline\alpha_j}\oplus\mg_{\overline\alpha_i-\overline\alpha_j}\\
                           & =   &\mg_{\overline{\alpha_i+\alpha_j}}\oplus\mg_{\overline{\alpha_i-\alpha_j}}\\
                           & =   & \text{span}_\C\{E_{ij}^++F_{ij}^+\ii\}\oplus \text{span}_\C\{E_{ij}^-+F_{ij}^-\ii\}.
\end{eqnarray*}
Thus,
$$2[E_i,E_j]\in\text{span}\{E_{ij}^+,F_{ij}^+,E_{ij}^-,F_{ij}^-\} = \ml_{ij}^+\oplus\ml_{ij}^-.$$
\end{proof}
\section{Exercises}
Unless specified otherwise, assume that $G$ is a compact Lie group with Lie algebra $\mg$, and $T\subset G$ is a maximal torus with Lie algebra $\tau\subset\mg$.
\begin{ex} For each $G\in\{SU(n),SO(2n),SO(2n+1),Sp(n)\}$, how many roots does $G$ have?
\end{ex}
\begin{ex}\label{dualrootex} For any linear function $\alpha:\tau\ra\R$, prove there exists a unique vector $\hat\alpha\in\tau$ such that for all $X\in\tau$, $\alpha(X)=\lb \hat\alpha,X\rb$.
\vspace{.05in}
\newline \emph{Hint: Define $\hat\alpha$ in terms of an orthonormal basis of $\tau$.}
\end{ex}
\begin{ex} If one begins with a different maximal torus of $G$, show that this does not effect the equivalence class of the system of dual roots of $G$ or the isomorphism class of $W(G)$.
\end{ex}
\begin{ex}\label{spex} Create tables describing all non-zero brackets of basis elements of $sp(n)$, as was done in this chapter for the other classical groups.
\end{ex}
\begin{ex} If $G=G_1\times G_2$, describe the roots and dual roots and Weil group of $G$ in terms of those of $G_1$ and $G_2$.
\end{ex}
\begin{ex} If $\hat\alpha$ is a dual root with root space $\ml=\text{span}\{E,F\}$, prove that $\text{span}\{E,F,\hat\alpha\}$ is a subalgebra of $\mg$ which is isomorphic to $su(2)$.
\begin{ex}\label{bellpeppers}
If $\{E,F\}$ is an ordered orthonormal basis of the root space $\ml$, then any other ordered orthonormal basis of $\ml$ will of be of the form $\{E'=L_gE,F'=L_gF\}$ for some $g\in O(2)$.  Show that $\hat\alpha=[E,F]$ and $\hat\alpha'=[E',F']$ are equal if and only if $g\in SO(2)$; otherwise $\hat\alpha'=-\hat\alpha$.
\end{ex}
\end{ex}
\begin{ex}
Prove that the center of $G$ equals the intersection of all maximal tori of $G$.
\end{ex}
\begin{ex}Define the \underline{centralizer} of $g\in G$ as $$C(g) := \{x\in G\mid xg=xg\}.$$  Let $C^0(g)$ denote the identity component of $C(g)$, as defined in Exercise~7.6.  Prove that $C^0(g)$ equals the union of all maximal tori of $G$ which contain $g$.
\vspace{.1in}\newline
\emph{HINT: if $x\in C^0(g)$, then $x$ belongs to a maximal torus of $C^0(g)$, which can be extended to a maximal torus of $G$.}
\end{ex}
\begin{ex}
If $a,b\in T$ are conjugate in $G$, prove that they are conjugate in $N(T)$.  That is, if $g\cdot a\cdot g^{-1}=b$ for some $g\in G$, prove that $h\cdot a\cdot h^{-1}=b$ for some $h\in N(T)$. \vspace{.1in}\newline\emph{HINT: If $g\cdot a\cdot g^{-1}=b$, then $T$ and $g\cdot T\cdot g^{-1}$ are two maximal tori of $C^0(b)$, so one is a conjugate of the other inside $C^0(b)$.  That is, there exists $x\in C^0(b)$ such that $xg\cdot T \cdot g^{-1} x^{-1}=T$. Now choose $h=xg$.}
\end{ex}
\begin{ex} When we studied double covers in Section 8.7, we claimed:
$$sp(1) \cong so(3),\,\,sp(1)\times sp(1) \cong so(4),\,\,\, sp(2)  \cong so(5),\,\,\, su(4)  \cong so(6).$$
For each of these Lie algebra isomorphism, show that the corresponding pair of Dynkin diagrams are identical.  Show that no other pair of Dynkin diagrams of classical groups is identical, and thus that there are no other classical Lie algebra isomorphims.
\end{ex}
\begin{ex}
Draw the root systems for the classical rank 2 groups: $SU(3)$, $SO(4)$, $SO(5)$, $Sp(2)$.  The only other rank 2 root system is pictured below:
\begin{figure}[h!]
   \scalebox{.15}{\includegraphics{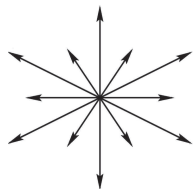}}
   \caption{The root system of the exceptional group $G_2$.}
   \end{figure}
\end{ex}
\begin{ex}
Prove that every rank 2 root system is one of the root systems from the previous exercise.
\vspace{.1in}\newline
\emph{Hint: The minimal angle, $\theta$, between any pair of roots must be $30^\circ$, $45^\circ$, $60^\circ$, or  $90^\circ$.  If $\alpha,\beta_1$ are roots which achieve this minimal angle, prove that for any integer $n$, there exists a root $\beta_n$ which forms an angle of $n\theta$ with $\alpha$, for example $\beta_2=-w_{\beta_1}\star\alpha$.  Show $|\beta_{n_1}|=|\beta_{n_2}|$ if $n_1$ and $n_2$ are either both odd or both even.}
\end{ex}
\begin{ex}\label{regvector} Prove Proposition~\ref{regularopendense}, which says that the strongly regular vectors of $G$ form an open dense subset of $\tau$.
\end{ex}
\begin{ex} Let $(\tau,R)$ be a root system and let $\alpha,\beta\in R$.  If $\angle(\alpha,\beta)$ is acute, prove that $\alpha-\beta\in R$.  If $\angle(\alpha,\beta)$ is obtuse, prove that $\alpha+\beta\in R$.
\vspace{.1in}\newline\emph{Hint: See the proof of Lemma~\ref{lobtuse}.  Note: From the dual root system of $G$, one can reconstruct its entire Lie algebra and bracket operation, which at least requires knowing which dual roots add or subtract to which dual roots.  This exercise give a glimpse of how such information can be obtained just from data about the angles between dual roots.}
\end{ex}
\begin{ex}
Prove that the Lie algebra of the center of $G$ equals
$$\mathfrak{z}(\mg):=\{A\in\mg\mid [A,X]=0\text{ for all }X\in\mg\},$$
which is called the \underline{center of $\mg$}.
\end{ex}
\begin{ex} For each of $G\in\{SU(n),SO(2n),SO(2n+1),Sp(n)\}$, check that $W(G)$ acts transitively on the set of dual roots of a fixed length.  That is, if $\hat\alpha,\hat\beta$ are dual roots with the same length, then there exists $w\in W(G)$ such that $w\star\hat\alpha=\hat\beta$.
\end{ex}
\begin{ex} Let $\Delta=\{\alpha_1,...,\alpha_t\}$ be a base of the root system $(\tau,R)$.  Prove that for any $w\in W(\tau,R)$, $w\star\Delta:=\{w\star\alpha_1,...,w\star\alpha_t\}$ is also a base of the root system.  It is also true that every base of the root system equals $w\star\Delta$ for some $w\in W(\tau,R)$.
\end{ex}
\begin{ex}\label{notuniqueroots}
If Lemma~\ref{notpar} were false, so that for example $\hat\alpha_1=\hat\alpha_2$, which is equivalent to $\alpha_1=\alpha_2$, show that this would allow multiple ways for the 4-dimensional space $\ml_1\oplus\ml_2$ to split into a pair of 2-dimensional $\Ad_T$-invariant spaces, for example:$$\ml_1\oplus\ml_2=\text{span}\{E_1+E_2,F_1+F_2\}\oplus\text{span}\{E_1-E_2,F_1-F_2\}.$$
Thus, the decomposition of Theorem~\ref{decompp} would not be unique, and we would therefore have extra roots and dual roots corresponding to the extra possible $\Ad_T$-invariant decompositions of $\mg$.
\end{ex}
\begin{ex}\label{complexalgebras1} An element of $so(n)_\C$ has the form $X=X_1+X_2\ii$ for some $X_1,X_2\in so(n)$.  Interpret such an $X$ as an element of $$so(n,\C):=\{A\in M_n(\C)\mid A+A^T=0\}\subset M_n(\C).$$
Via this interpretation, show that the complexified Lie bracket operation in $so(n)_\C$ becomes identified with the following operation in $so(n,\C)$: $[A,B]_\C=AB-BA.$
\end{ex}
\begin{ex}\label{complexalgebras2}\hspace{.1in}
\begin{enumerate}
\item Prove that every $X\in gl(n,\C)$ can be uniquely expressed as $X=X_1+X_2\ii$ for $X_1,X_2\in u(n)$.  Further, $X\in sl(n,\C)$ if and only if $X_1,X_2\in su(n)$.
\vspace{.1in}\newline\emph{Hint: $X = \frac{X-X^*}{2} +\frac{X+X^*}{2\ii}\ii.$}

\item Use the above decomposition to identify $u(n)_\C\cong gl(n,\C)$ and $su(n)_\C\cong sl(n,\C)$.  Show that the complex Lie bracket operations in $u(n)_\C$ and $su(n)_\C$ become identified with the operations in $gl(n,\C)$ and $sl(n,\C)$ defined as: $$[A,B]_\C=AB-BA.$$
\end{enumerate}
\end{ex}
\begin{ex}\label{subrootsystem}
Recall the injective function $\rho_n:M_n(\C)\ra M_{2n}(\R)$ from Chapter~2.
\begin{enumerate}
\item Show that $\rho_n(SU(n))$ is a subgroups of $SO(2n)$ which is isomorphic to $SU(n)$.
\item Show that $\rho_n(su(n))$ is a subalgebra of $so(2n)$ which is isomorphic to $su(n)$.
\item Show that $\lb \rho_n(X),\rho_n(Y)\rb = 2\cdot\lb X,Y\rb$ for all $X,Y\in M_n(\C)$, so $\rho_n$ provides an non-orthogonal equivalence between the system of dual root of $su(n)$ and a subsystem of the system of dual root of $so(2n)$.
\end{enumerate}
\end{ex}

\bibliographystyle{amsplain}

\begin{thebibliography}{9}

\bibitem{Baker} A. Baker, \emph{Matrix groups: an introduction to Lie group theory}. Springer, 2002.
\bibitem{Curtis} M. Curtis, \emph{Matrix groups, second edition}, Springer, 1975,1984.
\bibitem{erg} Cornfed, Formin, Sinai, \emph{Ergotic Theory}, Springer-Verlag, 1982.
\bibitem{FR} Frobenius, Journal fur die Reine und Angewandte Mathematik, 1878, Vol. 84.
\bibitem{Hall} B. Hall, \emph{Lie Groups, Lie Algebras, and Representations}, Springer, 2003.
\bibitem{Howe} R. Howe, \emph{Very basic Lie theory}, American Mathematical Monthly. \textbf{90} (1983), 600-623; Correction, Amer. Math. Monthly \textbf{91} (1984), 247.
\bibitem{Rossmann} W. Rossmann, \emph{Lie groups: an introduction through linear groups}, Oxford Science Publications, 2002.
\bibitem{Spivak} M. Spivak, \emph{A comprehensive introduction to differential geometry, Volume 1}, 1979.
\bibitem{Harvey}F.R. Harvey, \emph{Spinors and Calibrations}, Perspectives in Mathematics, Vol. 9. 1990.
\bibitem{Hel} S. Helgason, \emph{Differential Geometry, Lie Groups, and Symmetric Spaces}, 2000.
\bibitem{Warner} F. Warner, \emph{Foundations of differentiable manifolds and Lie groups}, 1983.
\bibitem{weeks} J. Weeks, \emph{The Poincar\'e Dodecahedral space and the mystery of the missing fluctuations}, Notices of the AMS, \textbf{51} (2004), number 6, pp. 610-619.
\end{thebibliography}

\printindex
\end{document}